\documentclass[a4paper,11pt]{article}

\usepackage{amsthm,amsmath,amssymb}
\usepackage[latin1]{inputenc}
\usepackage[english]{babel}
\usepackage{mathscinet}
\usepackage{graphics}
\usepackage{graphicx}
\usepackage[all,cmtip]{xy}
\usepackage{multirow}
\usepackage{hyperref}

\theoremstyle{plain}
\newtheorem{theo}{Theorem}[section]
\newtheorem{prop}[theo]{Proposition}

\newtheorem{lemma}[theo]{Lemma}

\theoremstyle{definition}
\newtheorem{defi}[theo]{Definition}
\newtheorem{ex}[theo]{Example}

\theoremstyle{remark}
\newtheorem{remark}[theo]{Remark}
\newtheorem{nothing}[theo]{\noindent\!\!\bf}

\DeclareMathOperator{\Sing}{Sing}
\DeclareMathOperator{\lcm}{lcm}
\DeclareMathOperator{\tr}{tr}
\DeclareMathOperator{\id}{Id}
\DeclareMathOperator{\Mat}{Mat}

\def\E{\mathcal{E}}
\def\C{\mathbb{C}}
\def\P{\mathbb{P}}
\def\N{\mathbb{N}}
\def\Z{\mathbb{Z}}
\def\Q{\mathbf{Q}}
\def\w{\omega}
\def\l{\ell}

\newcommand\bd{\mathbf{d}}

\newcommand\ba{\mathbf{a}}
\newcommand\bxi{\boldsymbol{\xi}}

\title{\bf Monodromy Zeta Function Formula for Embedded $\Q$-Resolutions}
\author{Jorge Mart\'{\i}n-Morales\footnote{Partially supported by the projects MTM2007-67908-C02-01, ``E15 Grupo Consolidado Geometr\'{\i}a'' from the goverment of Aragón, and FQM-333 from ``Junta de Andalucía''.}}
\date{Centro Universitario de la Defensa - IUMA.\\ Academia General Militar, Ctra.~de Huesca s/n.\\ 50090, Zaragoza, Spain.\\ jorge@unizar.es}

\begin{document}

\maketitle

\begin{abstract}
In a previous work we have introduced the notion of embedded $\Q$-resolution, which essentially consists in allowing the final ambient space to contain abelian quotient singularities. Here we give a generalization of N.~A'Campo's formula for the monodromy zeta function of a singularity in this setting. Some examples of its applications are shown.

\vspace{0.25cm}


\vspace{0.2cm}

\end{abstract}

\section*{Introduction}

In Singularity Theory, resolution is one of the most important tools. In the embedded case, the starting point is a singular hypersurface. After a sequence of
suitable blow-ups this hypersurface is replaced by a long list of smooth hypersurfaces (the strict transform and the exceptional divisors) which intersect in the simplest way (at any point one sees coordinate hyperplanes for suitable local coordinates). This process can be very expensive from the computational point of view and, moreover, only a few amount of the obtained data is used for the understanding of the singularity.

The experimental work shows that most of these data can be recovered if one allows some mild singularities to survive in the process (the quotient singularities). These partial resolutions, called \emph{embedded ${\Q}$-resolutions}, can be obtained as a sequence of weighted blow-ups and their computational complexity is extremely lower compared with standard resolutions. Moreover, the process is optimal in the sense that no useless data are obtained.

To do this, in~\cite{AMO11a}, we proved that Cartier and Weil divisors agree on $V$-manifolds. This allows one to develop a rational intersection theory on varieties with quotient singularities and study weighted blow-ups at points, see~\cite{AMO11b}. By using these tools we were able to get a big amount of information about the singularity.

In this paper we continue our study about $\Q$-resolutions. In particular, the behavior of the Lefschetz numbers and the zeta function of the monodromy with respect to an embedded $\Q$-resolution is investigated. These two invariants have already been studied in different contexts by several authors. Hence before going into details, let us recall some of those approaches.

Let $f:(\C^{n+1},0)\to (\C,0)$ be a germ of a non-constant analytic function and let $(H,0)$ be the hypersurface singularity defined by $f$. Consider $F = \{ x\in\C^{n+1} \,:\, ||x|| \leq \varepsilon,\, f(x) =~\!\!\eta \}$ ($0 < \eta << \varepsilon$, $\varepsilon$ small enough) the Milnor fiber and $h: F \to F$ the corresponding geometric monodromy. The induced automorphisms on the complex cohomology groups are denoted by $H^q(h) := h\,:\, H^q(F, \C) \to H^q(F, \C)$.

In~\cite{ACampo75}, A'Campo gives a method for computing the Lefschetz numbers of the iterates $h^k := h\, \circ \cdots \circ\, h$ of the geometric monodromy, defined by
$$
  \Lambda(h^k) := \sum_{q\geq 0} (-1)^q \tr H^q(h^k),
$$
in terms of an embedded resolution of the singularity $(H,0) \subset (\C^{n+1},0)$. These Lefschetz numbers are related to the monodromy zeta function
$$
  Z (f):= \prod_{q\geq 0} \det ( \id^{*} - t H^q(h) )^{(-1)^q}
$$
by the following well-known formula
\begin{equation}\label{zeta_numLef}
  Z (f) = \exp \Bigg(- \sum_{k \geq 1} \Lambda(h^k) \frac{t^k}{k} \Bigg).
\end{equation}
Using this relationship he derives a new expression for $Z(f)$. More precisely, let $\pi: X \to (\C^{n+1},0)$ be an embedded resolution of $(H,0)$. Consider
$$
  \pi^{*}(H) = \widehat{H} + \sum_{i=1}^r m_i E_i,
$$
the total transform of $H$, where $\widehat{H}$ is the strict transform of $H$ and $E_1,\ldots,E_r$ are the irreducible components of the exceptional divisor $\pi^{*}(0)$. Now, define
$$
  \check{E}_i := E_i \setminus \Bigg(E_i \cap \Big( \bigcup_{j\neq i} E_j \cup \widehat{H} \Big) \Bigg).
$$

Then, the Lefschetz numbers and the complex monodromy zeta function are given by
$$
  \Lambda(h^k) = \sum_{i=1,\ m_i|k}^r m_i\chi ( \check{E}_i ),\qquad
  Z(f) = \prod_{i=1}^r (1-t^{m_i})^{\chi( \check{E}_i )}.
$$
Thus the Euler characteristic of the Milnor fiber is $\chi(F) = \sum_{i=1}^r m_i \chi( \check{E}_i )$.

When $(H,0)$ defines an isolated singularity, both the characteristic polynomial of the monodromy $\Delta(t)$ and the Milnor number $\mu = \dim H^n(F,\C) = \deg \Delta(t)$ can be obtained from the zeta function as follows,
$$
  \Delta(t) = \left[ \frac{1}{t-1} \prod_{i=1}^r ( t^{m_i} - 1 )^{\chi( \check{E}_i )} \right]^{(-1)^n}, \quad \mu = (-1)^n \Big[ -1 + \sum_{i=1}^r m_i \chi( \check{E}_i ) \Big],
$$
and, in particular, $\mu = (-1)^n [ -1 + \chi(F) ]$ holds.

\vspace{0.25cm}

Another contribution in the same direction can be found in~\cite{GLM97}, where the authors give a generalization of A'Campo's formula for the monodromy zeta function via partial resolutions, that is, the map $\pi: X \to (\C^{n+1},0)$ is assumed to be just a modification (i.e.~the condition about normal crossing divisor in the embedded resolution is removed). Also Dimca, using the machinery of constructible sheaves, proved the same result allowing $X$ to be an arbitrary analytic space~\cite[Th.~6.1.14.]{Dimca04}.

\vspace{0.25cm}

The aim of this paper is to generalize all the results above, giving the corresponding A'Campo's formula and the Lefschetz numbers in terms of an embedded $\Q$-resolution, see Theorem~\ref{ATH} below. Note that Veys has already considered this problem for plane curve singularities~\cite{Veys97}.

\vspace{0.25cm}

Our plan is as follows. In \S\ref{prelim} some well-known preliminaries about quotient singularities and embedded $\Q$-resolutions are presented. The main result, i.e.~the generalization of A'Campo's formula in out setting, is stated and proven in \S\ref{mainTH} after having computed the monodromy zeta function of a divisor with $\mathbb{Q}$-normal crossings. In \S\ref{app_ex} weighted blow-ups are used to compute embedded $\Q$-resolutions in several examples, including a Yomdin-L\^{e} surface singularity, so as to apply the formula obtained. As a further application, the monodromy zeta function for not-well-defined functions giving rise to a zero set is introduced in \S\ref{not-well-defined}. Finally, in \S\ref{sec_non-abelian} it is exemplified the different behavior of A'Campo's formula using non-abelian groups, showing that ``double points'' in an embedded resolution may contribute to the monodromy zeta function.

\vspace{0.25cm}

As for notation, from now on and depending on the context, we shall denote the monodromy zeta function by $Z(f)$, $Z(f)(t)$, $Z(f;t)$, $Z_f(t)$ or $Z(t)$, interchangeably. The same applies for the Lefschetz numbers and the
characteristic polynomial.


\begin{center}
\begin{large}
\textbf{Acknowledgments}
\end{large}
\end{center}

This is part of my PhD thesis. I am deeply grateful to my advisors Enrique Artal and José Ignacio Cogolludo for supporting me continuously with their fruitful conversations and ideas.

This work was mainly written in Nice (France). I wish to express my gratitude to Alexandru Dimca for his insightful comments and discussions and to one of his students, Hugues Zuber, and all the members of the ``Laboratoire J.A. Dieudonné'' who made my stay more pleasant.

The author is partially supported by the Spanish projects MTM2010-2010-21740-C02-02, ``E15 Grupo Consolidado Geometría'' from the government of Aragón, and FQM-333 from ``Junta de Andalucía''.

\section{Preliminaries}\label{prelim}


Let us sketch some definitions and properties about $V$-manifolds, weighted projective spaces, embedded $\Q$-resolutions, and weighted blow-ups, see~\cite{AMO11a} and~\cite{AMO11b} for a more detailed exposition.

\subsection{$V$-manifolds and quotient singularities}

\begin{defi}
A $V$-manifold of dimension $n$ is a complex analytic space which admits an open
covering $\{U_i\}$ such that $U_i$ is analytically isomorphic to $B_i/G_i$ where
$B_i \subset \C^n$ is an open ball and $G_i$ is a finite subgroup of $GL(n,\C)$.
\end{defi}
 
The concept of $V$-manifolds was introduced in~\cite{Satake56} and they have the same homological
properties over $\mathbb{Q}$ as manifolds. For instance, they admit a Poincar{\'e}
duality if they are compact and carry a pure Hodge structure if they are compact
and Kähler, see~\cite{Baily56}. They have been classified locally by
Prill~\cite{Prill67}. It is enough to consider the so-called \emph{small subgroups} $G\subset GL(n,\C)$, i.e.~without
rotations around hyperplanes other than the identity.

\begin{theo}\label{th_Prill}{\rm (\cite{Prill67}).}~Let $G_1$, $G_2$ be small
subgroups of $GL(n,\C)$. Then $\C^n/G_1$ is isomorphic to $\C^n/G_2$ iff $G_1$ and $G_2$ are conjugate subgroups. $\hfill \Box$
\end{theo}

\begin{nothing}\label{notation_action}
For $\bd: = {}^t(d_1 \ldots d_r)$ we denote by
$\mu_{\bd} := \mu_{d_1} \times \cdots \times \mu_{d_r}$ a finite
abelian group written as a product of finite cyclic groups, that is, $\mu_{d_i}$
is the cyclic group of $d_i$-th roots of unity in $\C$. Consider a matrix of weight
vectors 
\begin{align*}
A & := (a_{ij})_{i,j} = [\ba_1 \, | \, \cdots \, | \, \ba_n ] \in \Mat (r \times n, \Z), \\ 
\ba_j & := {}^t (a_{1 j}\dots a_{r j}) \in \Mat(r\times 1,\Z),
\end{align*}
and the action
\begin{equation*}
\begin{array}{cr}
( \mu_{d_1} \times \cdots \times \mu_{d_r} ) \times \C^n  \longrightarrow  \C^n,&\bxi_\bd = (\xi_{d_1}, \ldots, \xi_{d_r}), \\[0.15cm]
\big( \bxi_{\bd} , \mathbf{x} \big) \mapsto (\xi_{d_1}^{a_{11}} \cdot\ldots\cdot
\xi_{d_r}^{a_{r1}}\, x_1,\, \ldots\, , \xi_{d_1}^{a_{1n}}\cdot\ldots\cdot
\xi_{d_r}^{a_{rn}}\, x_n ), & \mathbf{x} = (x_1,\ldots,x_n).
\end{array}
\end{equation*}

Note that the $i$-th row of the matrix $A$ can be considered modulo $d_i$. The
set of all orbits $\C^n / G$ is called ({\em cyclic}) {\em quotient space of
type $(\bd;A)$} and it is denoted by
$$
  X(\bd; A) := X \left( \begin{array}{c|ccc} d_1 & a_{11} & \cdots & a_{1n}\\
\vdots & \vdots & \ddots & \vdots \\ d_r & a_{r1} & \cdots & a_{rn} \end{array}
\right).
$$
The orbit of an element $\mathbf{x}\in \C^n$ under this action is denoted
by $[\mathbf{x}]_{(\bd; A)}$ and the subindex is omitted if no ambiguity
seems likely to arise. Using multi-index notation
the action takes the simple form
\begin{eqnarray*}
\mu_\bd \times \C^n & \longrightarrow & \C^n, \\
(\bxi_\bd, \mathbf{x}) & \mapsto & \bxi_\bd\cdot\mathbf{x}:=(\bxi_\bd^{\ba_1}\, x_1, \ldots, \bxi_{\bd}^{\ba_n}\, x_n).
\end{eqnarray*}
\end{nothing}

The quotient of $\C^n$ by a finite abelian group is always isomorphic to a quotient space of type $(\bd;A)$, see~\cite{AMO11a} for a proof of this classical result. Different types $(\bd;A)$ can give rise to isomorphic quotient spaces.

\begin{ex}\label{Ex_quo_dim1} When $n=1$ all spaces $X(\bd;A)$ are
isomorphic to~$\C$. It is clear that we can assume that $\gcd(d_i, a_{i})=1$.
If $r=1$, the map $[x] \mapsto x^{d_1}$ gives an isomorphism between $X(d_1; a_{1})$ and
$\C$. 

Consider the case $r=2$. Note that
$\C/(\mu_{d_1} \times \mu_{d_2})$ equals $(\C/\mu_{d_1})/ \mu_{d_2}$.
Using the previous isomorphism, it is isomorphic to $X(d_2, d_1 a_2)$,
which is again isomorphic to $\C$. By induction, we obtain the result for any~$r$.
\end{ex}

If an action is not free on $(\C^{*})^n$ we can factor the group by the kernel of the action and
the isomorphism type does not change. This motivates the following definition.

\begin{defi}\label{def_normalized_XdA}
The type $(\bd;A)$ is said to be {\em normalized} if the action is free on $(\C^{*})^n$
and~$\mu_\bd$ is small as subgroup of $GL(n,\C)$.
By abuse of language we often say the space $X(\bd;A)$ is written in a normalized
form when we mean the type $(\bd;A)$ is normalized.
\end{defi}

\begin{prop}
The space $X(\bd;A)$ is written in a normalized form if and only if the stabilizer
subgroup of $P$ is trivial for all~$P \in \C^n$ with exactly $n-1$ coordinates
different from zero.

In the cyclic case the stabilizer of a point as above (with exactly $\,n-1$
coordinates different from zero) has order $\gcd(d, a_1, \ldots, \widehat{a}_i,
\ldots, a_n)$.
\end{prop}

It is possible to convert general types~$(\bd;A)$ into their normalized form.
Theorem~\ref{th_Prill} allows one to decide whether two quotient spaces are
isomorphic. In particular, one can use this result to compute the singular points
of the space $X(\bd;A)$.
In Example~\ref{Ex_quo_dim1} we have explained this normalization process in dimension one.
The two and three-dimensional cases are treated in the following examples. 

\begin{ex}\label{X2}
All quotient spaces for $n=2$ are cyclic. The space $X(d;a,b)$ is written in a
normalized form if and only if $\gcd(d,a) = \gcd(d,b) = 1$. If this is not the
case, one uses the isomorphism\footnote{The notation $(i_1,\ldots,i_k) =
\gcd(i_1,\ldots,i_k)$ is used in case of complicated or long formulas.} (assuming
$\gcd(d,a,b)=1$) $X(d; a,b) \rightarrow X \big( \frac{d}{(d,a)(d,b)}; \frac{a}{(d,a)},
\frac{b}{(d,b)} \big)$, $[ (x,y) ] \mapsto [ (x^{(d,b)},y^{(d,a)}) ]$
to convert it into a normalized one.
\end{ex}

\begin{ex}\label{Ex_dim3} The quotient space $X(d;a,b,c)$ is
written in a normalized form if and only if $\gcd(d,a,b) = \gcd(d,a,c) =
\gcd(d,b,c) = 1$. As above, isomorphisms of the form $[(x,y,z)] \mapsto
[(x,y,z^k)]$ can be used to convert types $(d;a,b,c)$ into their normalized
form.
\end{ex}

In~\cite{Fujiki74} the author computes resolutions of cyclic quotient
singularities. In the two-dimensional case, the resolution process is due to Jung and Hirzebruch,
see~\cite{HNK71}.

\subsection{Weighted projective spaces}

The main reference that has been used in this section is \cite{Dolgachev82}.
Here we concentrate our attention on describing the analytic structure and singularities.

Let $\w:=(q_0,\ldots,q_n)$ be a weight vector, that is, a finite set of coprime positive
integers. 
There is a natural action of the multiplicative group $\C^{*}$ on
$\C^{n+1}\setminus\{0\}$ given by
$$
  (x_0,\ldots,x_n) \longmapsto (t^{q_0} x_0,\ldots,t^{q_n} x_n).
$$

The set of orbits $\frac{\C^{n+1}\setminus\{0\}}{\C^{*}}$ 
under this action is denoted by $\P^n_\w$ (or $\P^n(\w)$ in case of complicated weight vectors) 
and it is called the {\em weighted projective space} of type $\w$. 
The class of a nonzero element $(x_0,\ldots,x_n)\in \C^{n+1}$ 
is denoted by $[x_0:\ldots:x_n]_\w$ and the weight vector is omitted if no ambiguity seems likely to arise.
When $(q_0,\ldots,q_n)=(1,\ldots,1)$ one obtains the usual projective space
and the weight vector is always omitted. For $\mathbf{x}\in\C^{n+1}\setminus\{0\}$,
the closure of $[\mathbf{x}]_\w$ in $\C^{n+1}$ is obtained by adding the origin and it
is an algebraic curve.

\begin{nothing}\label{analytic_struc_Pkw}
\textbf{Analytic structure.} Consider the decomposition
$
\P^n_\w = U_0 \cup \cdots \cup U_n,
$
where $U_i$ is the open set consisting of all elements $[x_0:\ldots:x_n]_\w$
with $x_i\neq 0$. The map
$$
  \widetilde{\psi}_0: \C^n \longrightarrow U_0,\quad
\widetilde{\psi}_0(x_1,\cdots,x_n):= [1:x_1:\ldots:x_n]_\w
$$
defines an isomorphism $\psi_0$ if we replace $\C^n$ by
$X(q_0;\, q_1,\ldots,q_n)$.
Analogously, $X(q_i;\,q_0,\ldots,\widehat{q}_i,\ldots,q_n) \cong U_i$
under the obvious analytic map.
\end{nothing}

\begin{prop}[\cite{AMO11a}]\label{propPw}
Let $d_i := \gcd (q_0,\ldots,\widehat{q}_i,\ldots,q_n)$,
 $e_i:= d_0\cdot\ldots\cdot\widehat{d}_i\cdot\ldots\cdot d_n$
and $p_i:=\frac{q_i}{e_i}$. The following map is an isomorphism:
$$
\begin{array}{rcl}
\P^n \big(q_0,\ldots,q_n\big) & \longrightarrow & \P^n(p_0,\dots,p_n), \\[0.15cm]
\,[x_0:\ldots:x_n] & \mapsto &
\big[\,x_0^{d_0}:\ldots:x_n^{d_n}\,\big].
\end{array}
$$
\end{prop}

\begin{remark}
Note that, due to the preceding proposition, one can always assume the weight
vector satisfies $\gcd(q_0,\ldots,\widehat{q}_i,\ldots,q_n)=1$, for
$i=0,\ldots,n$. In particular, $\P^1{(q_0,q_1)} \cong \P^1$ and for $n=2$ we can
take $(q_0,q_1,q_2)$ pairwise relatively prime numbers. In higher dimension the situation
is a bit more complicated.
\end{remark}

\subsection{Embedded $\Q$-resolutions}

Classically an embedded resolution of $\{f=0\} \subset \C^n$ is a proper map
$\pi: X \to (\C^n,0)$ from a smooth variety $X$ satisfying, among other
conditions, that $\pi^{-1}(\{f=0\})$ is a normal crossing divisor. To weaken the
condition on the preimage of the singularity we allow the new ambient space $X$
to contain abelian quotient singularities and the divisor $\pi^{-1}(\{f=0\})$ to
have \emph{normal crossings} over this kind of varieties. This notion of normal
crossing divisor on $V$-manifolds was first introduced by Steenbrink
in~\cite{Steenbrink77}.

\begin{defi}
Let $X$ be a $V$-manifold with abelian quotient singularities. A hypersurface
$D$ on $X$ is said to be with {\em $\mathbb{Q}$-normal crossings} if it is
locally isomorphic to the quotient of a union of coordinate hyperplanes under a group
action of type $(\bd;A)$. That is, given $x \in X$, there is an isomorphism of
germs $(X,x) \simeq (X(\bd;A), [0])$ such that $(D,x) \subset (X,x)$ is identified
under this morphism with a germ of the form
$$
\big( \{ [\mathbf{x}] \in X(\bd;A) \mid x_1^{m_1} \cdot\ldots\cdot x_k^{m_k} = 0 \},
[(0,\ldots,0)] \big).
$$
\end{defi}

Let $M = \C^{n+1} / \mu_\bd$ be an abelian quotient space not necessarily cyclic
or written in normalized form. Consider $H \subset M$ an analytic subvariety of codimension one.

\begin{defi}\label{Qresolution}
An {\em embedded $\Q$-resolution} of $(H,0) \subset (M,0)$ is a proper analytic
map~$\pi: X \to (M,0)$ such that:
\begin{enumerate}
\item $X$ is a $V$-manifold with abelian quotient singularities.
\item $\pi$ is an isomorphism over $X\setminus \pi^{-1}(\Sing(H))$.
\item $\pi^{-1}(H)$ is a hypersurface with $\mathbb{Q}$-normal crossings on $X$.
\end{enumerate}
\end{defi}

\begin{remark}
Let $f:(M,0) \to (\C,0)$ be a non-constant analytic function germ. Consider
$(H,0)$ the hypersurface defined by $f$ on $(M,0)$. Let $\pi: X \to (M,0)$ be an
embedded $\Q$-resolution of $(H,0) \subset (M,0)$. Then $\pi^{-1}(H) = (f\circ
\pi)^{-1}(0)$ is locally given by a function of the form
$
x_1^{m_1}  \cdot\ldots\cdot x_k^{m_k} : X(\bd;A) \rightarrow \C.
$ 
\end{remark}

\subsection{Weighted blow-ups}

Weighted blow-ups can be defined in any dimension, see~\cite{AMO11a, AMO11b}. In this
section, we restrict our attention to the case $n=2$ and $n=3$.

\begin{nothing}\label{blowup_dimN_smooth}\textbf{Classical blow-up of $\C^{2}$.}
We consider
$$
\widehat{\C}^{2} := \big\{ ((x,y),[u:v]) \in \C^{2} \times \P^1 \mid
(x,y)\in \overline{[u:v]} \big\}.
$$
Then $\pi : \widehat{\C}^{2} \to \C^{2}$ is an isomorphism over
$\widehat{\C}^{2} \setminus \pi^{-1}(0)$. The {\em exceptional divisor} $E:=
\pi^{-1}(0)$ is identified with $\P^{1}$. The space
$\widehat{\C}^{2} = U_0 \cup U_1$ can be covered by $2$
charts each of them isomorphic to $\C^2$. For instance, the following map
defines an isomorphism:
\begin{eqnarray*}
\C^{2} & \longrightarrow & U_0 = \{ u \neq 0 \} \ \subset \
\widehat{\C}^{2},\\
(x,y) \ & \mapsto & \big( (x, x y),[1: y] \big).
\end{eqnarray*}
\end{nothing}

\begin{nothing}\label{weighted_blowup_dimN_smooth}\textbf{Weighted $(p,q)$-blow-up of $\C^{2}$.} 
Let~$\w = (p,q)$ be a weight vector with coprime entries. As above, consider the space
$$
\widehat{\C}^{2}_{\w} := \big\{ ((x,y),[u:v]_{\w}) \in \C^{2} \times
\P^1_{\w} \mid (x,y) \in \overline{[u:v]}_{\w}
\big\}.
$$
It can be covered by $\widehat{\C}^2_{\w} = U_1 \cup U_2 = X(p;-1,q) \cup X(q;\,p,-1)$ and the charts are given by
$$
\begin{array}{c|rcl}
\text{First chart} & X(p;-1,q) & \longrightarrow & U_1, \\[0.10cm]
& \,[(x,y)] & \mapsto & ((x^p,x^q y),[1:y]_{\w}). \\ \multicolumn{2}{c}{} \\
\text{Second chart} & X(q;\,p,-1) & \longrightarrow & U_2, \\[0.10cm]
& \,[(x,y)] & \mapsto & ((x y^p, y^q),[x:1]_{\w}).
\end{array}
$$
The exceptional divisor $E:=\pi_{\w}^{-1}(0)$ is isomorphic to $\P^1_{\w}$ which
is in turn isomorphic to~$\P^1$ under the map
$
[x:y]_{\w} \longmapsto [x^q:y^p]$.
The singular points of $\widehat{\C}^2_{\w}$ are cyclic quotient singularities located at the exceptional divisor. They actually coincide with the origins of the two charts and they are written in their normalized form.
\end{nothing}

\begin{nothing}\textbf{Weighted $(p,q,r)$-blow-up of $\C^{3}$.}
Let $\pi := \pi_{\w}: \widehat{\C}^3_{\w} \to \C^3$ be the weighted blow-up at the origin with respect to $\w=(p,q,r)$,
$\gcd(\w)=1$. The new space is covered by three open sets
$$
\widehat{\C}^3_{\w} = U_1 \cup U_2 \cup U_3 = X(p;-1,q,r) \cup X(q;p,-1,r) \cup X(r;p,q,-1),
$$
and the charts are given by
\begin{equation*}
\begin{array}{cc}
X(p;-1,q,r) \longrightarrow U_1: & [(x,y,z)] \mapsto ((x^p, x^q y, x^r z),[1:y:z]_{\w}), \\[0.25cm]
X(q;p,-1,r) \longrightarrow U_2: & [(x,y,z)] \mapsto ((x y^p,y^q,y^r z),[x:1:z]_{\w}), \\[0.25cm]
X(r;p,q,-1) \longrightarrow U_3: & [(x,y,z)] \mapsto ((x z^p, y z^q, z^r),[x:y:1]_{\w}).
\end{array}
\end{equation*}

In general $\widehat{\C}^3_{\w}$ has three lines of (cyclic quotient) singular
points located at the three axes of the exceptional divisor $\pi^{-1}_{\w}(0) \simeq \P^2_{\w}$.
Namely, a generic point in $x=0$ is a cyclic point of type
$\C\times X(\gcd(q,r);p,-1)$.
Note that although the quotient spaces are written in their normalized form,
the exceptional divisor can be simplified:
$$
\begin{array}{rcl}
\P^2(p,q,r) & \longrightarrow & \P^2 \displaystyle\left(\frac{p}{(p,r)\cdot
(p,q)},\frac{q}{(q,p)\cdot (q,r)},
\frac{r}{(r,p)\cdot (r,q)}\right),\\[0.5cm]
\displaystyle \,[x:y:z] & \mapsto & [x^{\gcd(q,r)}:y^{\gcd(p,r)}:z^{\gcd(p,q)}].
\end{array} 
$$

Using just a weighted blow-up of this kind, one can find an embedded $\Q$-resolution for Brieskorn-Pham surfaces singularities, i.e.~$x^a + y^b + z^c = 0$, see Example~\ref{Brieskorn-Pham_surfaces}.
\end{nothing}

\section{Statement and Proof of the Main Theorem}\label{mainTH}

This section is devoted to the generalization of A'Campo's formula for embedded $\Q$-resolution.

One way to proceed is to rebuild A'Campo's paper~\cite{ACampo75}, thus giving a model of the Milnor fibration in our setting. This method is very natural but perhaps a bit long and tedious. In~\cite{GLM97}, the authors give a generalization of A'Campo's formula for the monodromy zeta function via partial resolution but the ambient space considered there is still smooth and the proof can not be generalized to an arbitrary analytic variety.

That is why a very general result by Dimca is used instead, see Theorem~\ref{th_dimca} below. This leads us to talk about constructible complexes of sheaves with respect to a stratification and also about the nearby cycles associated with an analytic function. Using this theorem only the monodromy zeta function of a monomial defining a function over a quotient space of type $(\bd;A)$ is needed.

\subsection{A result by Dimca}

To state the result we need some notions about sheaves and constructibility. We refer for instance to~\cite{Dimca04} and the references listed there for further details.

Consider $Sh(X,\text{Vect}_\C)$ the abelian category of sheaves of $\C$-vector spaces on a topological space~$X$. To simplify notation its derived category is often denoted by $D^{*}(X)$. The constant sheaf corresponding to $\C$ is denoted by $\underline{\C}_X$; it is by definition the sheaf associated with the constant presheaf that sends every open subset of $X$ to $\C$. If $U \subset X$ is connected open then $\underline{\C}_X (U) = \C$.

Let $f: X \to Y$ be a continuous mapping between two topological spaces. The direct image functor $f_{*}: Sh(X,\text{Vect}_\C) \rightarrow Sh(Y,\text{Vect}_\C)$ is defined on objects by $(f_{*} \mathcal{F})(V) = \mathcal{F}(f^{-1}(V))$, for any sheaf $\mathcal{F}$ on~$X$ and any open set $V \subset Y$.
This functor is additive and left exact; its derived functor is denoted by $R f_{*} : D^{*}(X) \to D^{*}(Y)$.

The inverse image functor $f^{-1}: Sh(Y,\text{Vect}_\C) \rightarrow Sh(X,\text{Vect}_\C)$ is defined as $f^{-1} \mathcal{G}$ being the sheaf associated with the presheaf
$$
  U \longmapsto \lim_{\begin{smallmatrix} \longrightarrow \\ f(U) \subset V \end{smallmatrix}} \mathcal{G}(V).
$$
Here $\mathcal{G}$ is a sheaf on $Y$ and $U\subset X$ is open. This functor is exact and hence the corresponding derived functor $Rf^{-1}: D^{*}(Y) \to D^{*}(X)$ is usually denoted again by $f^{-1}$.

If $f(U) \subset Y$ is open then $(f^{-1} \mathcal{G})(U) = \mathcal{G}(f(U))$. In particular, if $i_{U}: U \hookrightarrow X$ denotes the inclusion of an open set, then $i_{U}^{-1} \mathcal{F} = \mathcal{F}|_U$. The restriction to an arbitrary subspace $Z \subset X$ is defined by $\mathcal{F}|_Z := i_{Z}^{-1} \mathcal{F}$, where $i_Z: Z \hookrightarrow X$ is the inclusion.
Using this notation one has $\underline{\C}_X|_Z := i^{-1}_Z \underline{\C}_X = \underline{\C}_Z$.

Let $X$ be a complex analytic space and $\mathcal{S}=\{X_j\}_{j\in J}$ a locally finite partition of $X$ into non-empty, connected, locally closed subsets called {\em strata} of $\mathcal{S}$. The partition $\mathcal{S}$ is called a {\em stratification} if it satisfies the following conditions.

\begin{enumerate}
\item The boundary condition, i.e.~each boundary $\partial X_j = \overline{X_j}\setminus X_j$ is a union of strata in $\mathcal{S}$.
\item Constructibility, i.e.~for all $j\in J$ the spaces $\overline{X_j}$ and $\partial X_j$ are closed complex analytic subspaces in $X$.
\item Stratification, i.e.~all the strata are smooth constructible subvarieties of $X$.
\end{enumerate}

\begin{defi}
Let $\mathcal{S} = \{ X_j \}_{j\in J}$ be a stratification on $X$.

(i) A sheaf complex $\mathcal{F}^\bullet \in D^{*}(X)$ is called {\em $\mathcal{S}$-constructible} if the restriction of each cohomology sheaf $\mathcal{H}^q ( \mathcal{F}^{\bullet} )|_{X_j}$ is a $\underline{\C}_{X_j}$-local system of finite rank, that is, one has the isomorphisms of $\underline{\C}_{X_j}$-vector spaces $\mathcal{H}^q ( \mathcal{F}^{\bullet} )|_{X_j} \ \simeq \ \underline{\C}_{X_j}^{r_{j,q}}$.

(ii) Given $u: \mathcal{F}^{\bullet} \to \mathcal{F}^{\bullet}$ an automorphisms of $\underline{\C}_X$-vector spaces, the complex $\mathcal{F}^{\bullet}$ is called {\em equivariantly $\mathcal{S}$-constructible} with respect to $u$, if it is $\mathcal{S}$-constructible and the induced automorphisms on the cohomology groups $\mathcal{H}^q (u)_x : \mathcal{H}^m ( \mathcal{F}^{\bullet} )_x \to \mathcal{H}^m ( \mathcal{F}^{\bullet} )_x$ are all conjugate.
\end{defi}

Let $X$ be a complex analytic variety and $g:X\to \C$ a non-constant analytic function. Consider the diagram,
$$
\xymatrix{
g^{-1}(0) \ar@{^{(}->}[r]^>>>>>{i} & X & \ar @{} [dr] |{\#} X \setminus g^{-1}(0) \ar[d]_{f} \ar@{_{(}->}[l]_>>>>>{j} & E \ar[d]^{\hat{f}} \ar[l]_>>>>>{\hat{\pi}} \\
&& \C^{*} & \widetilde{\C}^{*} \ar[l]^{\exp}
}
$$
where $i: g^{-1}(0) \hookrightarrow X$ and $j: X\setminus g^{-1}(0) \hookrightarrow X$ are inclusions, $\widetilde{\C}^{*}$ is the universal cover of~$\C^{*}$, and $E$ denotes the pull-back.

\begin{defi}
Let $\mathcal{F}^\bullet \in D^{*}(X)$ be a complex. The {\em nearby cycles} of $\mathcal{F}^\bullet$ with respect to the function $g: X \to \C$ is defined to be the sheaf complex given by
$$
  \psi_g \mathcal{F}^\bullet := i^{-1} R(j\circ \hat{\pi})_{*} (j\circ \hat{\pi})^{-1} \mathcal{F}^\bullet \ \in \ D^{*} ( g^{-1}(0) ).
$$
\end{defi}

The nearby cycles is a local operation in the sense that if $U \subset X$ is an open set, then~$(\psi_g \mathcal{F}^\bullet)|_W = \psi_{g|_W} \mathcal{F}^\bullet|_W$ holds. Also, note that $\psi_g \mathcal{F}^\bullet$ only depends on $g$ and $\mathcal{F}^\bullet|_{X\setminus g^{-1}(0)}$.

There is an associated monodromy deck transformation $h: E \to E$ coming from the action of the natural generator of $\pi_1 (\C^{*})$ which satisfies $\hat{\pi} \circ h = \hat{\pi}$. This homeomorphism induces an isomorphism of complexes
$$
  M: \psi_g \mathcal{F}^\bullet \longrightarrow \psi_g \mathcal{F}^\bullet.
$$

For every point $x\in g^{-1}(0)$ there is a natural isomorphism from the stalk cohomology of $\psi_g \mathcal{F}^\bullet$ at $x$ to the cohomology of the Milnor fiber at $x$ with coefficients in $\mathcal{F}^\bullet$, that is, for all $\epsilon > 0$ small enough and all $t \in \C^{*}$ with $|t| << \epsilon$, one has
\begin{equation}\label{isom_nearby_cycles}
\mathcal{H}^q (\psi_g \mathcal{F}^\bullet)_x \simeq \mathbb{H}^q ( g^{-1}(t) \cap B_\epsilon(x), \mathcal{F}^\bullet_{|}) \simeq \mathbb{H}^q ( g^{-1}(t) \cap \overline{B_\epsilon(x)}, \mathcal{F}^\bullet_{|}),
\end{equation}
where the open ball $B_\epsilon(x)$ is taken inside any local embedding of $(X,x)$ in an affine space.

The monodromy morphism $M_x$ on the left-hand side corresponds to the morphism on the right-hand side induced by the monodromy homeomorphism of the local Milnor fibration associated with $g:(X,x) \to (\C,0)$.

Now we are ready to state Dimca's theorem. To be precise he only considered the case when the ambient space is smooth $M = \C^{n+1}$, see below. Repeating exactly the same arguments one obtains the result for any analytic variety.

\begin{theo}[\cite{Dimca04}, Th.~6.1.14]\label{th_dimca}
Let $f:(M,p) \to (\C,0)$ be the germ of a non-constant analytic function which is defined on a small neighborhood $U$ of $p$. Let $H$ be the hypersurface $\{ x \in U \mid f(x) = 0 \}$. Assume $\pi: X \to U$ is a proper analytic map such that $\pi$ induces an isomorphism between $X\setminus \pi^{-1} (H)$ and $U\setminus H$.

Let $g = f\circ \pi$ denote the composition and $j: X\setminus \pi^{-1}(H) \hookrightarrow X$ the inclusion.
Let $\mathcal{S}$ be a finite stratification of the exceptional divisor $\pi^{-1}(0)$ such that $\psi_g \big( R j_{*} \underline{\C}_{\,X\setminus \pi^{-1}(H)} \big)$ is equivariantly $\mathcal{S}$-constructible with respect to the semisimple part of $M$. Then,
$$
  \Lambda(h) = \sum_{S\in \mathcal{S}} \chi(S) \Lambda(g,x_S)\,; \qquad
  Z(f) = \prod_{S\in \mathcal{S}} Z(g,x_S)^{\chi(S)},
$$
where $x_S$ is an arbitrary point in the stratum $S$ and $Z(g,x_S)$, $\Lambda(g,x_s)$ are the zeta function and the Lefschetz number of the germ $g$ at $x_S$.
\end{theo}

\begin{remark}
Let $\mathcal{F}^\bullet = Rj_{*} \underline{\C}_{X\setminus \pi^{-1}(H)}$. Using the notation of the previous theorem the isomorphism of~(\ref{isom_nearby_cycles}) tells us that $\mathcal{H}^q (\psi_g \mathcal{F}^\bullet)_x = H^q (F_x, \C)$ where~$F_x$ is the Milnor fiber at $x$. This clarifies when the complex of sheaves $\psi_g \mathcal{F}^\bullet$ is equivariantly $\mathcal{S}$-constructible with respect to the semisimple part of $M$. In particular, this condition is satisfies for instance when the local equation of $g$ along each stratum is the same.
\end{remark}

\subsection{Monodromy zeta function of a normal crossing divisor}

Let $M = \C^{n}/\mu_\bd$ be a quotient space of type $X(\bd; A)$, not necessarily cyclic or written in a normalized form. Recall the multi-index notation.
$$
  X(\bd;A) = X \left( \begin{array}{c|ccc} d_1 & a_{11} & \ldots & a_{1n}\\ \vdots & \vdots & \ddots & \vdots \\ d_r & a_{r1} & \ldots & a_{rn} \end{array} \right),\quad \begin{array}{c} \bd = (d_1,\ldots,d_r), \\[0.1cm] \ \ba_j = (a_{1j},\ldots,a_{rj}). \end{array}
$$

In Section~\ref{prelim}, cf.~Example~\ref{Ex_quo_dim1}, we have seen that for each $j=1,\ldots,n$ there is an isomorphism
\begin{equation}\label{iso_dim_1_l}
\begin{array}{rcl}
X(\bd; \ba_j) & \longrightarrow & \C \\[0.15cm]
\,[x_j] & \mapsto & x_j^{\l_j},
\end{array}
\end{equation}
where 
\begin{equation*}
\displaystyle  \l_j = \lcm \left( \frac{d_1}{\gcd(d_1,a_{1j})}, \ldots, \frac{d_r}{\gcd(d_r,a_{rj})} \right).
\end{equation*}

Given a homogeneous polynomial defined over $M$ the classical formula for the monodromy zeta function depending on the degree of the polynomial and the Euler characteristic of the Milnor fiber seems to be more complicated in this setting. Using resolution of singularities,
one can provide formulas at least for plane curves and surfaces but the trick of applying the fixed point theorem does not work anymore. However, for our purpose, only the normal crossing case is needed.

Note that the zeta function and the Lefschetz numbers also exist in case of singular underlying spaces, such as $X{(\bd;A)}$. Moreover, if the function $f$ is defined by a quasi-homogeneous polynomial, then $f: X{(\bd;A)}\setminus f^{-1}(0) \to \C^{*}$ is a locally trivial fibration and the global Minor fibration is equivalent to the local one.

\vspace{0.25cm}

We first proceed to compute the geometric monodromy of a homogeneous polynomial $f: M \to \C$ of degree $N := \deg(f)$. Let $\alpha:[0,1] \to \C^{*}$ be a generator of the fundamental group of $\C^{*}$, for example, $\alpha(t) = \exp(2\pi i t)$ and consider $[{\bf x}] \in F = f^{-1}(1)$. The path
\begin{eqnarray*}
\widetilde{\alpha} :\, [0,1] & \longrightarrow & M \setminus f^{-1}(0) \\
t & \mapsto & \big[ ( e^{\frac{2\pi i}{N} t} x_1, \ldots, e^{\frac{2\pi i}{N} t} x_n) \big],
\end{eqnarray*}
defines a lifting of~$\alpha$ with initial point $[(x_1,\ldots,x_n)]$. Thus the geometric monodromy $h: F \to F$ corresponds to the map
$$
  \widetilde{\alpha}(0) = \big[ (x_1,\ldots,x_n) \big] \ \stackrel{h}{\longmapsto} \ \big[ ( e^{\frac{2\pi i}{N}} x_1, \ldots, e^{\frac{2\pi i}{N}} x_n) \big] = \widetilde{\alpha}(1).
$$

As in the case $M=\C^n$, this also works for quasi-homogeneous polynomials, replacing the exponentials for suitable numbers according to the weights.

\vspace{0.25cm}

Let us study the monodromy zeta function in the simplest normal crossing case, i.e.~$f = x_1^{m_1}: M \to \C$. The Milnor fiber
$$
F := f^{-1}(1) = \{ [{\bf x}] \in M \mid x_1^{m_1} = 1 \}
$$
has the same homotopy type as $F' := \{ [(x_1, 0, \ldots,0)]\in M \mid x_1^{m_1} = 1 \}$, which can be identified with
$$
 \{ [x_1] \in X(\bd; \ba_1) \mid x_1^{m_1} = 1 \}.
$$

In fact, $r: F \to F' : [{\bf x}] \mapsto [x_1]$ is a strong deformation retraction. Since $h(F') \subset F'$, the geometric monodromy $h:F \to F$ is homotopic to its restriction $h':= h|_{F'}: F' \to F'$.
Using the isomorphism \eqref{iso_dim_1_l},
$$
X(\bd; \ba_1) \simeq \C: [x] \mapsto x^{\l_1},
$$
the claim is reduced to the calculation of the zeta function of the polynomial $x_1^{m_1/\l_1} : \C \to \C$. But this is known to be~$1-t^{m_1/\l_1}$.

\vspace{0.25cm}

Assume now that $f = x_1^{m_1} \cdots x_k^{m_k} : M \to \C$, $k \geq 2$. The Milnor fiber $F:=f^{-1}(1)$ has the same homotopic type as
$$
  F':= \Big\{ \big[ (x_1, \ldots x_k) \big] \in \frac{S^1 \times \stackrel{(k)}{\cdots} \times S^1}{\mu_\bd} \ \ \big| \ \ x_1^{m_1} \cdot \ldots \cdot x_k^{m_k} = 1 \Big\},
$$
where $\mu_\bd$ defines an action of type $(\bd; \ba_1, \ldots, \ba_k)$ on the space $(S^1)^k$. As above, there is a strong deformation retraction
$$
r\,:\, F \longrightarrow F', \quad
\,[{\bf x}] \mapsto \Big[ \big( \frac{x_1}{|x_1|}, \ldots, \frac{x_k}{|x_k|}, 0, \ldots,0 \big) \Big],
$$
that satisfies $h(F')\subset F'$.
We shall see that the Lefschetz numbers $\Lambda((h')^j) = \Lambda(h^j) = 0$ for all $j\geq 1$. This would imply $Z_f(t) = 1$ by virtue of~(\ref{zeta_numLef}). Two cases arise.
\begin{itemize}
\item If $(h')^j$ does not have fixed points, then by the fixed point theorem $\Lambda((h')^j) = 0$.
\item Otherwise $(h')^j$ is the identity map and $\Lambda((h')^j) = \chi(F') = 0$.
\end{itemize}

Note that there is an unramified covering
$$
(S^1)^k \supset \widetilde{F'} := \{ x_1^{m_1} \cdot \ldots \cdot x_k^{m_k} = 1\} \stackrel{\pi}{\longrightarrow} F'
$$
with a finite number of sheets. The first of the preceding spaces $\widetilde{F'}$ has $e = \gcd(m_1,\ldots,m_k)$ disjoint components, each of them homotopically equivalent to a real $(k-1)$-dimensional torus~$\mathcal{T}_{k-1} = (S^1)^{k-1}$. It follows that
$$
\chi(F') = \frac{1}{\deg \pi}\, e\, \chi(\mathcal{T}_{k-1}) = 0.
$$

Note that the condition $k\geq 2$ has only been used at the end. In the case $k=1$, one has
$$
\deg \pi = \l_1, \quad e=m_1, \quad \chi(\mathcal{T}_0) = 1, \quad \chi(F') = m_1/\l_1.
$$

We summarize the previous discussion in the following lemma.

\begin{lemma}\label{lemma_zetaMon}
The monodromy zeta function of a normal crossing divisor given by $x_1^{m_1} \cdot \ldots \cdot x_k^{m_k}: X(\bd; A) \rightarrow \C$, $k \geq 1$, is
$$
Z \left( x_1^{m_1} \cdot \ldots \cdot x_k^{m_k}: X(\bd; A ) \rightarrow \C;\, t \right) =
\begin{cases}
1-t^{\frac{m_1}{\l_1}} & k=1 ; \\
1 & k\geq 2,
\end{cases}
$$
where $\displaystyle  \l_1 = \lcm \left( \frac{d_1}{\gcd(d_1,a_{11})}, \ldots, \frac{d_r}{\gcd(d_r,a_{r1})} \right)$.
\end{lemma}

\subsection{A'Campo's formula for embedded $\Q$-resolutions}

Let $f: (M,0) \to (\C,0)$ be a non-constant analytic function germ and let $(H,0) \subset (M,0)$ be the hypersurface defined by $f$. Given an embedded ${\bf Q}$-resolution of $(H,0)$, $\pi: X \to (M,0)$, consider as in the classical case,
$$
  \check{E}_i := E_i \setminus \Bigg(E_i \cap \Big( \bigcup_{ \begin{subarray}{c} k=1,\ldots,s \\ k\neq i \end{subarray} } E_k \cup \widehat{H} \Big) \Bigg),
$$
where $E_1, \ldots, E_s$ are the irreducible components of the exceptional divisor of $\pi$, and $\widehat{H}$ is the strict transform of $H$.

\begin{defi}
Let $X$ be a complex analytic space having only abelian quotient singularities and consider $D$ a $\mathbb{Q}$-divisor with normal crossings on~$X$. Let $q\in D$ be a point living in exactly one irreducible component of $D$. Then, the equation of~$D$ at~$q$ is given by a function of the form $x_j^{m}: X(\bd; A) \rightarrow \C$, where $x_j$ is a local coordinate of $X$ in~$q$.

The {\em multiplicity} of $D$ at $q$, denoted by $m(D,q)$, is defined by
$$
m(D,q):= \frac{m}{\l_j}, \qquad \l_j = \lcm \left( \frac{d_1}{\gcd(d_1,a_{1j})},\ldots,\frac{d_r}{\gcd(d_r,a_{rj})} \right).
$$

If there exists $T$ contained in exactly one irreducible component of~$D$ and the function $q\in T \mapsto m(D,q)$ is constant, then we use the notation $m(T) := m(D,q)$, where $q\in T$ is an arbitrary point.
\end{defi}

\begin{remark}
The integer $m(D,q)$ does not depend on the type $(\bd;A)$ representing the quotient space. A more general definition, including the case when $q \in D$ belongs to more than one irreducible component, will be given in a future work. \end{remark}

To simplify the notation one writes $E_0 = \widehat{H}$ and $S=\{0,1,\ldots,s\}$ so that the stratification of $X$ associated with the $\mathbb{Q}$-normal crossing divisor $\pi^{-1}(H) = \bigcup_{i \in S} E_i$ is defined by setting
\begin{equation}\label{strat_emb_res}
  E_{I}^\circ := \Big( \cap_{i \in I} E_i \Big) \setminus \Big( \cup_{i\notin I} E_i \Big),
\end{equation}
for a given possibly empty set $I\subseteq S$. Note that, for $i=1,\ldots,s$, one has that $E_{\{i\}}^{\circ} = \check{E}_i$.


Let $X = \bigsqcup_{j\in J} Q_j$ be a finite stratification on $X$ given by its quotient singularities so that the local equation of $g = f\circ \pi\,$ at $q \in E_I^{\circ} \cap Q_j$ is of the form
$$
  x_1^{m_1} \cdot \ldots \cdot x_k^{m_k}:\, B/G \longrightarrow \C,
$$
where $B$ is an open ball around $q$, and $G$ is an abelian group acting diagonally as in $(\bd;A)$. The multiplicities $m_i$'s and the action $G$ are the same along each stratum $E_I^{\circ} \cap Q_j$, i.e.~it does not depend on the chosen point $q \in E_I^{\circ} \cap Q_j$. Let us denote
$$
\check{E}_{i,j} := \check{E}_i \cap Q_j, \qquad m_{i,j} := m(\check{E}_{i,j}).
$$


The following result is nothing but a generalization of Theorem~\ref{ATH2} written in the language of divisors. To use the classical convection on indices $M= \C^{n+1}/\mu_\bd$ (instead of $\C^n/\mu_{\bf d}$) in the theorem below.

\begin{theo}\label{ATH}
Let $f:(M,0) \to (\C,0)$ be a non-constant analytic function germ and let $H = \{ f = 0\}$. Consider $F$ the Milnor fiber and $h:F \to F$ the geometric monodromy. Assume $\pi:X \to (M,0)$ is an embedded $\Q$-resolution of $(H,0)$. Then, using the notation above, one has: ($i=1,\ldots,s$,\, $j\in J$)
\begin{enumerate}
\item The Lefschetz number of $\,h^k = h\, \circ \stackrel{(k)}{\cdots} \circ\,
h: F \to F$, $k\geq 0$, and the Euler characteristic of $F$ are
$$
  \displaystyle \Lambda(h^k) = \sum_{i,j,\ k|m_{i,j}} m_{i,j} \cdot \chi( \check{E}_{i,j} ), \hspace{0.5cm} \chi(F) = \sum_{i,j} m_{i,j} \cdot \chi( \check{E}_{i,j} ) = \Lambda(h^0).
$$
\item The local monodromy zeta function of $f$ at $0$ is
$$
  \displaystyle Z(t) = \prod_{i,j} \left(1-t^{m_{i,j}}
  \right)^{\chi( \check{E}_{i,j} )}.
$$
\item In the isolated case, the characteristic polynomial of the complex monodromy of $(H,0) \subset (M,0)$~is
$$
  \Delta(t) = \left[\frac{1}{t-1} \prod_{i,j} \left(t^{m_{i,j}}-1
  \right)^{\chi( \check{E}_{i,j} )}\right]^{(-1)^n},
$$
and the Milnor number is $\displaystyle\mu = (-1)^n
\Big[-1+\sum_{i,j} m_{i,j} \cdot \chi( \check{E}_{i,j}) \Big]$. 
\end{enumerate}
\end{theo}

\begin{proof}
Only the proof of (2) is given; the other items follow from this one.
Using that $E_0 = \widehat{H}$ and $S = \{ 0,1,\ldots,s \}$, the support of the total transform can be written as
$$
\pi^{-1}(H) = \widehat{H} \cup \pi^{-1}(0) = \bigcup_{i \in S} E_i.
$$

Let $X = \bigsqcup_{I \subseteq S} E_{I}^{\circ}$ be the stratification of $X$ given in~(\ref{strat_emb_res}) associated with this $\mathbb{Q}$-normal crossing divisor.
This partition gives rise to a stratification on $\pi^{-1}(0) = \bigsqcup E_{I}^{\circ}$, where the intersection is taken over $$I\in \mathcal{P}(S) \setminus \{ \emptyset, \{ 0 \} \}.$$

However, the equivariant property is not satisfied in general, since the strata may contain singular points of $X$. Instead, let $\mathcal{S}$ be the following finer stratification
$$
  \mathcal{S} = \Big\{ E_{I}^{\circ} \cap Q_j \Big\}_{\begin{subarray}{c} I\subset S,\,\, j\in J\\ I\neq\, \emptyset,\, \{0\} \end{subarray}}.
$$

Now the family $\mathcal{S}$ is a finite stratification of the exceptional divisor of $\pi$ such that the complex $\psi_{f\circ \pi} ( Rj_{*} \underline{\C}_{\,X\setminus \pi^{-1}(H)} )$ is equivariantly $\mathcal{S}$-constructible, where $$j: X\setminus \pi^{-1}(H) \lhook\joinrel\longrightarrow X$$ is the inclusion. Hence Theorem~\ref{th_dimca} applies. Moreover, given $q\in \pi^{-1}(0)$ there exist $I=\{i_1,\ldots,i_k\} \subset S$, $k\geq 1$ ($k=1 \Rightarrow i_1 \neq 0$), and $j\in J$ such that the local equation of $g = f\circ \pi$ at $q$ is given by the function
$$
x_{i_1}^{m_{i_1}} \cdot \ldots \cdot x_{i_k}^{m_{i_k}} :\, B_j/G_j \longrightarrow \C.
$$

The numbers $m_{i_j}$'s and the action $G_j$ are the same along each stratum of~$\mathcal{S}$. By Lemma~\ref{lemma_zetaMon}, the strata with $k \geq 2$ do not contribute to the monodromy zeta function.

Take $x_T = x_{I,j}$ an arbitrary point in $E_I^\circ \cap Q_j$, then from the previous discussion one has
\begin{align*}
  Z(f) & = \prod_{T \in \mathcal{S}} Z(g, x_T) = \prod_{\begin{subarray}{c} I\subset S,\,\, j\in J\\ I\neq\, \emptyset,\, \{0\} \end{subarray}} Z(g, x_{I,j})^{\chi(E_I^\circ \cap Q_j)} \\ & = \prod_{\begin{subarray}{c} i = 1,\ldots,s \\ j\in J \end{subarray}} Z(g, x_{\{i\},j})^{\chi(E_{\{i\}}^{\circ} \cap Q_j)} = \prod_{\begin{subarray}{c} i = 1,\ldots,s \\ j\in J \end{subarray}} (1 - t^{m_{i,j}})^{\chi( \check{E}_{i,j} )}.
\end{align*}

Above, Lemma~\ref{lemma_zetaMon} is used again for the computation of the monodromy zeta function at~$x_{\{i\},j}$. Observe also that $E_{\{i\}}^\circ \cap Q_j = \check{E}_{i,j}$. Now the proof is complete.
\end{proof}

This theorem has already been proven by Veys in~\cite{Veys97} for plane curve singularities, that is, for $n=1$. If all $d$'s are equal to one, then $\pi:X\to (\C^{n+1},0)$ is an embedded resolution of $(H,0)$ in the classical sense and one obtains exactly the formula by A'Campo~\cite{ACampo75}.

\begin{remark}
Let $X = \bigsqcup_{j\in J} Q'_j$ be another finite stratification of $X$ such that the function $q \in \check{E}_{i} \cap Q'_j \longmapsto m(E_{i},q)$ is constant. Then, the previous theorem still holds replacing $\check{E}_{i,j} = \check{E}_i \cap Q_j$ by $\check{E}_i \cap Q'_j$.
\end{remark}

\begin{remark}\label{simplification_norm}
When $\Sing(M) \subset H$ then $M\setminus H$ is smooth and thus so is $X\setminus \pi^{-1}(H)$. Consequently, all singularities of $X$ are contained in the total transform $\pi^{-1}(H)$ and the numbers $m_{i,j}$'s take the simple form
$$
m_{i,j} = \frac{m}{\lcm(d_1,\ldots,d_r)},
$$
after having normalized the types involved in the corresponding embedded $\Q$-resolution of the singularity, cf.~Remark~\ref{easier_dprime}.
\end{remark}

\section{Applications and Examples}\label{app_ex}

The following result is nothing but a reformulation of Theorem~\ref{ATH}, adopted to the situation, which is encountered in the examples.

\begin{theo}\label{ATH2}
Let $f:(\C^{n+1},0) \to (\C,0)$ be a non-constant analytic function germ defining an isolated singularity and let $H = \{ f = 0\}$. Assume $\pi:X \to (\C^{n+1},0)$ is an embedded $\Q$-resolution of $(H,0)$, having $X$ cyclic quotient singularities. Let $X_0 = \pi^{-1}(H)$ be the total transform and $S=\pi^{-1}(0)$ the exceptional divisor. Consider $S_{m,d'}$ to be the set
$$
\left\{\begin{tabular}{l|l}
\multirow{3}{*}{$s\in S$\ } &\ the local equation of $X_0$ in $s$ is given by the well-defined
\\ &\ function $x_i^m: X{(d;a_0,\ldots,a_n)} \to \C$, where $x_i$ is a local \\ &\ coordinate of $X$ in $s$, and $d / \gcd(d,a_i) = d'$.
\end{tabular}\right\}.
$$

Then, the characteristic polynomial of the complex monodromy of the hypersurface $(H,0)$ is
\begin{equation}\label{Acampo_equation}
  \Delta(t) = \Bigg[ \frac{1}{t-1} \prod_{m,d'} ( t^{ m/d' } - 1 )^{\chi(
S_{m,d'} )} \Bigg]^{(-1)^n}.
\end{equation}
\end{theo}

\begin{remark}\label{easier_dprime}
If all cyclic quotient singularities appearing in $X$ are written in their normalized form and $\gcd(d,a_i) \neq 1$, then the space $X\setminus X_0$ must contain singular points. This, however, contradicts that $\pi$ is an embedded $\Q$-resolution. Therefore after normalizing, one can always assume that $d=d'$, cf.~Remark~\ref{simplification_norm}.
\end{remark}

\begin{ex}\label{ex_onePuiseuxPair}
Let $f:\C^2 \to \C$ be the function given by $f=x^p+y^q$ and assume that $e=\gcd(p,q)$, $p=p_1 e$ and $q=q_1 e$. Consider $\pi: \widehat{\C}^2{(q_1,p_1)} \to \C^2$ the weighted blow-up at the origin of type $(q_1,p_1)$. Recall that $\widehat{\C}^2{(q_1,p_1)} = U_0\cup U_1$ has two singular points corresponding
to the origin of each chart.

In $U_0 = X(q_1;-1,p_1)$ the total transform of $f$ is given by the function $x^{p_1 q_1 e} (1+y^q)$. The equation $y^q=-1$ only has $q/q_1=e$ different solutions in $U_0$ and the local equation of the total transform at each of theses points is of the form $x^{p_1 q_1 e}\, y$.

Hence the proper map $\pi$ is an embedded $\Q$-resolution of ${\bf C}=\{f=0\}$ where all spaces are written in their normalized form.

\begin{figure}[h t]
\centering\includegraphics{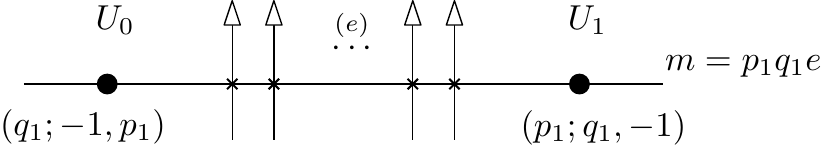}
\caption{Embedded $\Q$-resolution of $f = x^p + y^q$.}
\end{figure}

The set $S_{m,d}$ is not empty for $(m,d)= (p_1 q_1 e,1)$, $(p_1 q_1 e, q_1)$, $(p_1 q_1 e, p_1)$. Their Euler characteristics are
$$
\chi(S_{p_1 q_1 e,1}) = 2-(e+2)=-e, \qquad \chi(S_{p_1 q_1 e, q_1}) =
\chi(S_{p_1 q_1 e, p_1}) = 1.
$$

Now, we apply Theorem~\ref{ATH2} and obtain
$$
  \Delta(t) = \frac{(t-1)(t^{\frac{p q}{e}}-1)^e}{(t^{p}-1)(t^{q}-1)}.
$$

Another interesting way to calculate the characteristic polynomial could be the following. Consider $\pi: \widehat{\C}^2(q,p) \to \C$ the blow-up at the origin of type $(q,p)$. Now, $U_0 = X(q;-1,p)$ and the equation of the total transform in this chart is $x^{p q} (1+y^q)$. As above, the map $\pi$ is an embedded $\Q$-resolution of ${\bf C}$ and our formula can be applied. However, the exceptional divisor, outside the two singular points, is not given by $x^{p q}$ as one can expect at first sight. The reason is that $X(q;-1,p)$ is not written in a normalized form.

The isomorphism $X(q;-1,p) \cong X(q_1;-1,p_1)$ sends the function $x^{p q}:X(q;-1,p)\to \C$ to $x^{\frac{p q}{e}}: X(q_1;-1,p_1) \to \C$, and thus the required equation is $x^{\frac{p q}{e}} : \C^2 \to \C$. After applying the formula one obtains the same characteristic polynomial.

\begin{figure}[h t]
\centering
\includegraphics{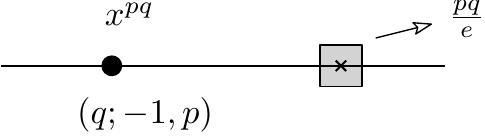}
\caption{Non-normalized cyclic quotient singularity.}
\end{figure}

This example shows that although one can blow up using non coprime weights, if possible, it is better to do it with the corresponding coprime weights to simplify calculations. However, the normalized condition is not necessary in the hypothesis of the statement.
\end{ex}

\begin{ex}
Assume $p_1/q_1 < p_2/q_2$ are two irreducible fractions and $\gcd(q_1,q_2) = 1$. Let~${\bf C}$ be the complex plane curve with Puiseux expansion
$$
y=x^{\frac{p_1}{q_1}}+x^{\frac{p_2}{q_2}}.
$$

Consider $\pi_1: \widehat{\C}^2(q_1,p_1) \to \C^2$ the weighted blow-up at the origin of type $(q_1,p_1)$. The exceptional divisor $\E_0$ has multiplicity $p_1 q_1 q_2$ and contains two singular points of type $(q_1;-1,p_1)$ and $(p_1;q_1,-1)$. The strict transform $\widehat{{\bf C}}$ of the curve and $\E_0$ intersect at one smooth point, say $P$. The Puiseux expansion of $\widehat{{\bf C}}$ in a small neighborhood of this point is
$$
  y = x^\frac{p_2 q_1 - p_1 q_2}{q_2},
$$
and thus $\pi_1$ is not a $\Q$-resolution.

\begin{figure}[h t]
\centering
\includegraphics{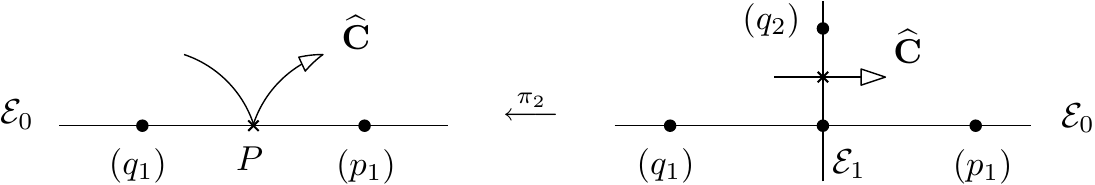}
\caption{Embedded $\Q$-resolution of ${\bf C} = \{ y=x^{\frac{p_1}{q_1}}+x^{\frac{p_2}{q_2}} \} $.}
\label{fig_acampo3}
\end{figure}

Now let $\pi_2$ be the weighted blow-up at $P$ of type $(q_2,p_2 q_1 - p_1 q_2)$. The multiplicity of the new exceptional divisor $\E_1$ is $q_2(p_1 q_1 q_2 + p_2 q_1 - p_1 q_2)$. It intersects transversally $\E_0$ at a singular point of type $$(p_2 q_1 - p_1 q_2;\, q_2,-1)$$ and also contains another singular point of type $(q_2;-1,p_2 q_1)$. The strict transform of the curve is a smooth variety and cuts transversally $\E_1$ at a smooth point.

Hence the composition $\pi_1\circ\pi_2$ defines an embedded $\Q$-resolution of ${\bf C}\subset \C^2$ where all cyclic quotient spaces are written in their normalized form. Figure~\ref{fig_acampo3} illustrates the whole process.

The corresponding Euler characteristics are $\chi(\E_0\setminus\{3\text{ points}\}) = \chi(\E_1\setminus\{3\text{ points}\}) = -1$ and $\chi = 1$ for the three singular points. Note that the singular point of type $(p_2 q_1 - p_1 q_2)$ does not contribute to the monodromy zeta function, since it belongs to more than one divisor. After applying formula~\eqref{Acampo_equation}, one obtains
$$
  \Delta(t) = \frac{\big(t-1\big)\big(t^{p_1 q_1 q_2}-1\big)\big(t^{q_2(p_1 q_1 q_2 + p_2 q_1 - p_1 q_2)}-1\big)}{\big(t^{p_1 q_2}-1\big)\big(t^{q_1 q_2}-1\big)\big(t^{p_1 q_1 q_2 + p_2 q_1 - p_1 q_2}-1\big)},\qquad \mu = \deg \Delta(t).
$$

In case $q_1$ and $q_2$ are not coprime, the same arguments apply and one can find a formula for the characteristic polynomial of an irreducible plane curve with two (and then with arbitrary) Puiseux pairs. These formulas are quite involved and we omit them.
\end{ex}

\begin{ex}\label{ex_chi}
Let $e_1, e_2, e_3$ be three positive integers and denote $e=\gcd(e_1,e_2,e_3)$. Assume that $\w=(\frac{e_1}{e},\frac{e_2}{e},\frac{e_3}{e})$ is a weight vector of pairwise relatively prime numbers. Let ${\bf C}$ be the projective curve in~$\P^2_\w$ defined by the polynomial
$$
  F=x^{\frac{e_2 e_3}{e}}+y^{\frac{e_1 e_3}{e}}+z^{\frac{e_1 e_2}{e}}.
$$
Note that this polynomial is quasi-homogeneous of degree $e_1 e_2 e_3/e^2$. One is interested in computing the Euler characteristic of ${\bf C}$.

Consider $\pi: \widehat{\C}^3_{\w} \to \C^3$ the weighted blow-up at the origin with respect to $\w$ and take the affine variety $H=\{F=0\}\subset\C^3$. The space~$\widehat{\C}^3_{\w} = U_0 \cup U_1 \cup U_2$ has just three singular points, corresponding to the origin of each chart, and located at the exceptional divisor $E=\pi^{*}(0)\cong \P^2_\w$. The order of the cyclic groups are $\frac{e_3}{e}$, $\frac{e_2}{e}$ and $\frac{e_1}{e}$ respectively.

In the third chart $U_2 = X(\frac{e_3}{e};\frac{e_1}{e},\frac{e_2}{e},-1)$ the equation of the total transform is
$$
  z^{\frac{e_1 e_2 e_3}{e^2}}(x^{\frac{e_2 e_3}{e}}
  +y^{\frac{e_1 e_3}{e}}+1).
$$
One sees that the exceptional divisor and the strict transform are smooth varieties intersecting transversally. Thus $\pi$ is an embedded $\Q$-resolution of $H$ where all the quotient spaces are written in a normalized form.

\begin{figure}[h t]
\centering
\includegraphics{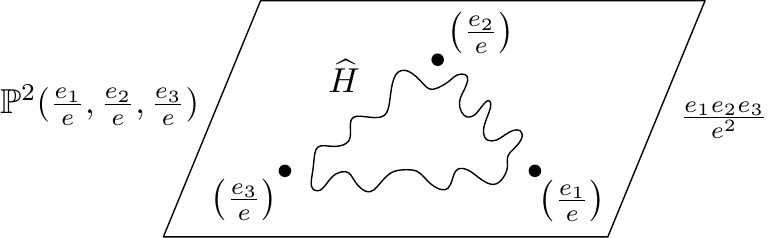}
\caption{Embedded $\Q$-resolution of $F = x^{\frac{e_2 e_3}{e}} + y^{\frac{e_1 e_3}{e}} + z^{\frac{e_1 e_2}{e}}$.}
\end{figure}

The set $S_{m,d}$ is not empty for $m=e_1 e_2 e_3/e^2$ and $d\in\{1,\frac{e_1}{e}, \frac{e_2}{e}, \frac{e_3}{e}\}$. Since the intersection
$E\cap \widehat{H}$ can be identified with ${\bf C}$, the Euler characteristics are
$$
\chi(S_{m,1})=-\chi({\bf C}), \qquad \chi(S_{m,\frac{e_1}{e}}) =
\chi(S_{m,\frac{e_2}{e}}) =
\chi(S_{m,\frac{e_3}{e}}) = 1.
$$

From Theorem~\ref{ATH2}, the characteristic polynomial of $H$ is
$$
  \Delta(t) = \frac{\big(t^{\frac{e_1 e_2}{e}}-1\big)\big(t^{\frac{e_1
e_3}{e}}-1\big)
  \big(t^{\frac{e_2 e_3}{e}}-1\big)}
  {\big(t-1\big)\big(t^{\frac{e_1 e_2 e_3}{e^2}}-1\big)^{\chi({\bf C})}}.
$$

On the other hand, the Milnor number is well-known to be $\mu = (\frac{e_1 e_2}{e}-1)
(\frac{e_1 e_3}{e}-1)(\frac{e_2 e_3}{e}-1)$. Using that $\mu = \deg \Delta(t)$
one finally obtains
$$
  \chi({\bf C}) = e_1 + e_2 + e_3 - \frac{e_1 e_2 e_3}{e}.
$$
\end{ex}

\begin{ex}\label{Brieskorn-Pham_surfaces}
Let $p,q,r$ be three positive integers and consider $f:\C^3\to \C$ the polynomial function given by
$$
 f=x^p+y^q+z^r.
$$

To simplify notation we set $e_1=\gcd(q,r)$, $e_2=\gcd(p,r)$, $e_3 = \gcd(p,q)$, $e=\gcd(p,q,r)$, and $k = e_1 e_2 e_3$. The following information will be useful later.
$$
\begin{array}{l c l}
\multicolumn{3}{c}{\gcd(q r,p r,p q) = \displaystyle\frac{e_1 e_2 e_3}{e} =
\frac{k}{e},}\\[0.3cm]
\displaystyle d_1:= \gcd \left(\frac{e p r}{k},\frac{e p q}{k}\right) = \frac{e p}{e_2
e_3}\,;
& \ & \displaystyle a_1:=\lcm (d_2,d_3) = \frac{e^2q r}{e_1 k} = d_2
d_3\,,\\[0.35cm]
\displaystyle d_2:= \gcd \left(\frac{e q r}{k},\frac{e p q}{k}\right) = \frac{e q}{e_1
e_3}\,;
& \ & \displaystyle a_2:=\lcm (d_1,d_3) = \frac{e^2p r}{e_2 k}\,,\\[0.35cm]
\displaystyle d_3:= \gcd \left(\frac{e q r}{k},\frac{e p r}{k}\right) = \frac{e r}{e_1
e_2}\,;
& \ & \displaystyle a_3:=\lcm (d_1,d_2) = \frac{e^2p q}{e_3 k}\,.\\
\end{array}
$$

Take the weight vector $\w = \frac{e}{k} (q r,p r,p q)$ and let $\pi:\widehat{\C}^3_\w \to \C^3$ be the weighted blow-up at the origin with respect to $\w$. The new space $\widehat{\C}^3_{\w} = U_0 \cup U_1 \cup U_2$ has three lines (each of them isomorphic to $\mathbb{P}^1$) of singular points located at the exceptional divisor $E=\pi^{-1}(0)\cong\mathbb{P}^2_\w$. They actually coincide with the three lines $L_0, L_1, L_2$ at infinity of~$\mathbb{P}^2_\w$.

\begin{figure}[h t]
\centering
\includegraphics{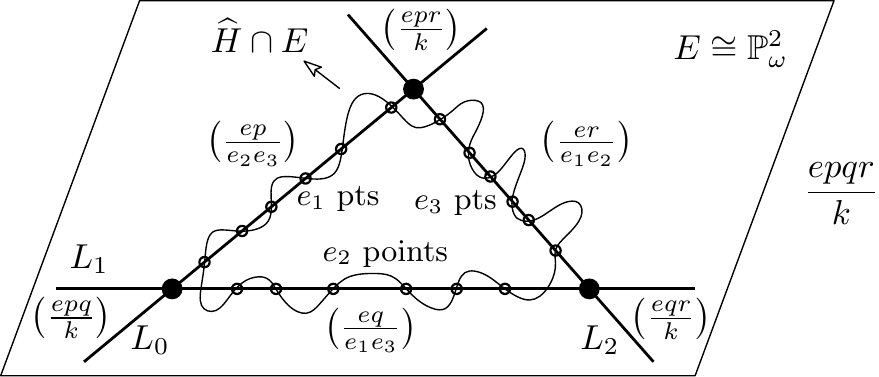}
\caption{Embedded $\Q$-resolution of $f=x^p+y^q+z^r$.}
\end{figure}

In the third chart $U_2=X(\frac{e p q}{k}; \frac{e q r}{k},\frac{e p r}{k},-1)$, an equation of the total transform is
$$
z^{\frac{e p q r}{k}}(x^p+y^q+1),
$$
where $z=0$ is the exceptional divisor and the other equation corresponds to the strict transform.

Working in this coordinate system, one sees that the line $L_0$ (resp.~$L_1$) and $\widehat{H}$ intersect at exactly $e_1$ (resp.~$e_2$) points. Analogously, $L_2\cap\widehat{H}$ consists of $e_3$ points. Moreover, one has that $\widehat{H}$ and $E$ are smooth varieties that intersect transversally. Hence the map $\pi$ is an embedded $\Q$-resolution of $\{f=0\} \subset \C^3$ where all the cyclic quotient spaces are presented in normalized form.

The Euler characteristics as well as the fractions $m/d$ for the nonempty sets $S_{m,d}$ are calculated in the table below.
$$
\begin{array}{c|c|c|c|c|}\cline{2-5}
&&&&\\[-0.2cm]
& S_{\frac{e p q r}{k},1} & S_{\frac{e p q r}{k},\frac{e p}{e_2 e_3}} &
S_{\frac{e p q r}{k},\frac{e q}{e_1 e_3}} & S_{\frac{e p q r}{k},\frac{e r}{e_1
e_2}} \\[0.3cm] \hline
\multicolumn{1}{|c|}{}&&&& \\[-0.2cm]
\multicolumn{1}{|c|}{\displaystyle\frac{m}{d}} & \displaystyle\frac{e p q r}{k}
&
\displaystyle\frac{q r}{e_1} &
\displaystyle\frac{p r}{e_2} & \displaystyle\frac{p q}{e_3} \\[0.3cm]
\hline
\multicolumn{1}{|c|}{\chi} & \begin{array}{c} e_1+e_2+e_3\\ -\chi({\bf
C})\end{array}
& -e_1 & -e_2 & -e_3 \\[0.0cm] \hline
\end{array}\vspace{0.25cm}
$$
$$
\begin{array}{c|c|c|c|}\cline{2-4}
&&&\\[-0.2cm]
& S_{\frac{e p q r}{k},\frac{e q r}{k}} & S_{\frac{e p q r}{k},\frac{e p r}{k}} &
S_{\frac{e p q r}{k},\frac{e p q}{k}} \\[0.3cm] \hline
\multicolumn{1}{|c|}{}&&& \\[-0.3cm]
\multicolumn{1}{|c|}{m/d} & p & q & r \\[0.2cm]
\hline
\multicolumn{1}{|c|}{}&&& \\[-0.3cm]
\multicolumn{1}{|c|}{\chi} & 1 & 1 & 1 \\[0.1cm] \hline
\end{array}\vspace{0.25cm}
$$

Here we denote by ${\bf C}$ the variety in $\mathbb{P}^2_{\w}$ defined by the $\w$-homogeneous polynomial $x^p+y^q+z^r$. Recall that the map $\P^2_{\w} \rightarrow \P^2{(\frac{e_1}{e},\frac{e_2}{e},\frac{e_3}{e})}$ given by
$$
  [x:y:z]_{\w} \longmapsto [x^\frac{e p}{e_2 e_3}: y^\frac{e q}{e_1 e_3}:
  z^\frac{e r}{e_1 e_2}]_{(\frac{e_1}{e},\frac{e_2}{e},\frac{e_3}{e})}
$$
is an isomorphism and maps the hypersurface ${\bf C}$ to $\{x^{\frac{e_2 e_3}{e}}+y^{\frac{e_1 e_3} {e}}+z^{\frac{e_1 e_2}{e}}=0\}$. By the preceding example its Euler characteristic is
$$
  \chi({\bf C}) = e_1 + e_2 + e_3 - \frac{e_1 e_2 e_3}{e},
$$
and finally, from Theorem~\ref{ATH2}, one obtains the characteristic polynomial of~$f$,
$$
  \Delta(t) = \frac{\big(t^\frac{e p q r}{e_1 e_2 e_3}-1\big)^\frac{e_1 e_2
e_3}{e}
  \big(t^p-1\big)\big(t^q-1\big)\big(t^r-1\big)}
  {\big(t-1\big)\big(t^\frac{q r}{e_1}-1\big)^{e_1}
  \big(t^\frac{p r}{e_2}-1\big)^{e_2} \big(t^\frac{p q}{e_3}-1\big)^{e_3}}.
$$
\end{ex}

Note that the Euler characteristic of ${\bf C}$ could also be obtained using that the Milnor number is $\mu = (p-1)(q-1)(r-1) = \deg\Delta(t)$, as in the previous example.

\begin{ex}
Let $f:\C^3\to \C$ be the polynomial function defined by $f = z^{m+k} + h_m(x,y,z)$. Assume that ${\bf C} = \{h_m = 0\} \subseteq \P^2$ has only one singular point $P = [0:0:1]$, which is locally isomorphic to the cusp $x^q+y^p$, $\gcd(p,q)=1$. Denote $k_1 = \gcd(k,p)$ and $k_2 = \gcd(k,q)$.

Consider the classical blow-up at the origin $\pi_1:\widehat{\C}^3 \to \C^3$. In the third chart, the local equation of the total transform is
$$
  z^m (z^k + x^q + y^p) = 0.
$$
The strict transform $\widehat{H}$ and the exceptional divisor $E_0$ intersect transversally at every point but in $P\in {\bf C} \equiv E_0\cap \widehat{H}$. Also $\widehat{H}\setminus P$ is smooth.

One is therefore interested in the blowing-up at the point $P$ with respect to $(k p,k q,p q)$. However, in order to obtain cyclic quotient spaces in normalized form, it is more suitable to choose $\w = (\frac{k p}{k_1 k_2}, \frac{k q}{k_1 k_2}, \frac{p q}{k_1 k_2})$ instead. Let $\pi_2$ be the weighted blow-up at $P$ with respect to the vector $\w$. The local equation of the total transform in the second chart is given by
$$
  \left\{ y^{\frac{p q}{k_1 k_2}(m+k)} z^m (z^k + x^q + 1) = 0\right\} \subset
  X\left(\frac{k q}{k_1 k_2}; \frac{k p}{k_1 k_2}, -1, \frac{p q}{k_1
k_2}\right),
$$
where $y=0$ represents the new exceptional divisor $E_1$.

The composition $\pi = \pi_1 \circ \pi_2$ is an embedded $\Q$-resolution. The final situation is illustrated in Figure~\ref{fig_acampo6}.

\begin{figure}[h t]
\centering
\includegraphics{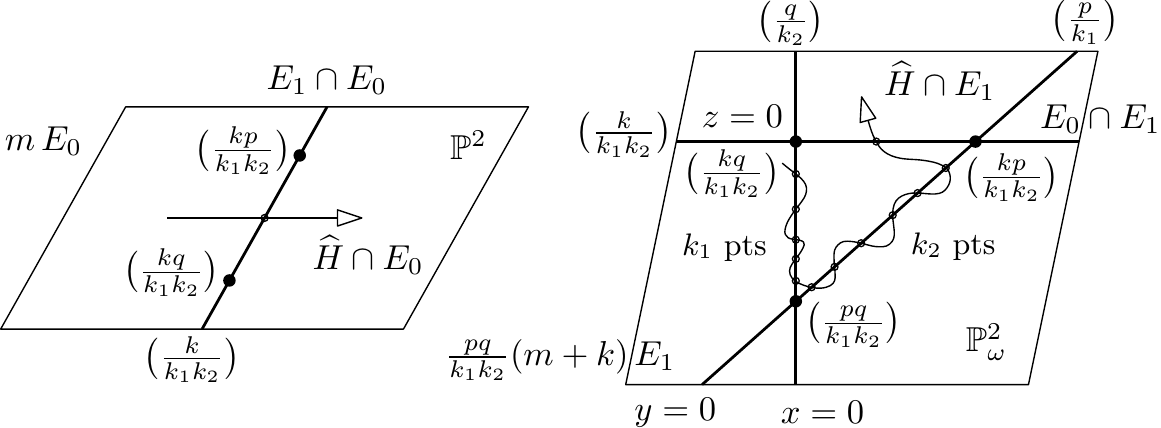}
\caption{Intersection of $E_0$ (resp.~$E_1$) with the rest of components.}
\label{fig_acampo6}
\end{figure}

The sets for which the Euler characteristic has to be computed are
$$
  S_{m,1},\quad S_{\l,1},\quad S_{\l,\frac{p}{k_1}},\quad S_{\l,\frac{q}{k_2}},
  \quad S_{\l,\frac{p q}{k_1 k_2}};\qquad \l= \frac{p q}{k_1 k_2}(m+k).
$$
Clearly $\chi(S_{\l,\,p q/k_1 k_2}) = 1$, $\chi(S_{\l,\,p/k_1})= - k_2$ and $\chi(S_{\l,\,q/k_2}) = - k_1$, since they are homeomorphic to a point, $\P^1\setminus\{k_2+2 \text{ points}\}$ and $\P^1\setminus\{k_1+2 \text{points}\}$ respectively. The set $S_{m,1}$ is $\P^2\setminus {\bf C}$. Finally, we use the additivity of the Euler characteristic to compute $\chi(S_{\l,1})$.

Indeed, let ${\bf D}\subset \P^2(k_1,k_2,1)$ be the projective variety defined by the equation $z^{k_1 k_2} + x^{k_2} + y^{k_1} = 0$. Note that ${\bf D}$ is isomorphism to
$$
\widehat{H}\cap E_1 = \{ z^k + x^q + y^p = 0 \} \subset \P^2_{\w}
$$
and, by Example~\ref{ex_chi} (using $e_1=k_1$, $e_2=k_2$, $e_3=1$), its Euler characteristic is $k_1 + k_2 + 1 - k_1 k_2$. Then,
$$
  \chi(S_{\l,1}) = 3 - (2+2+2+\chi({\bf D})) + k_1 + k_2 + 4 = k_1 k_2.
$$

Every cyclic quotient singularity is written in a normalized form and thus the generalized A'Campo's formula can be applied with $d'=d$,
\begin{align*}
  \Delta(t) & = \frac{\big(t^m-1\big)^{\chi(\P^2\setminus{\bf C})}}{t-1} \cdot
  \frac{\big(t^{m+k}-1\big)\big(t^{\frac{p q}{k_1 k_2}(m+k)}-1\big)^{k_1 k_2}}
  {\big(t^{\frac{p}{k_1}(m+k)}-1\big)^{k_1}
  \big(t^{\frac{q}{k_2}(m+k)}-1\big)^{k_2}} \\ & =
  \frac{\big(t^m-1\big)^{\chi(\P^2\setminus{\bf C})}}{t-1}
\cdot  \Delta^k_{P}(t^{m+k}).
\end{align*}

Let us explain the notation. The symbol $\Delta_P(t)$ denotes the characteristic polynomial of ${\bf C}$ at $P=[0:0:1]$, where the curve is locally isomorphic to $x^q+y^p$, and if $\Delta (t) = \prod_i (t^{m_i}-1)^{a_i})$, then $\Delta^k(t)$ denotes
$$
  \Delta^k(t) = \prod_i (t^{\frac{m_i}{\gcd(m_i,k)}} - 1)^{\gcd(m_i,k) a_i}.
$$
\end{ex}

\begin{remark}
Using these techniques an embedded $\Q$-resolution associated with the family of examples $z^{m+k} + h_m(x,y,z)$, where $h_m$ defines an arbitrary projective curve in $\P^2$ such that $\Sing(h_m)\cap \{ z = 0\} = \emptyset$ in $\P^2$, can be computed. In particular, the formulas by D.~Siersma~\cite{Siersma90} and J.~Stevens~\cite{Stevens89} for the characteristic polynomial of Yomdin-L\^{e} surface singularities can be obtained in this way. They is not presented explicitly because it not the purpose of this paper.
Note that this family of singularities has also been extensively studied by E.~Artal \cite{Artal94} and I.~Luengo \cite{Luengo87}.
\end{remark}

We conclude this section by emphasizing that in the classical A'Campo's formula one has to pay attention to compute the Euler characteristic while the multiplicities remain trivial. Using our formula we also have to take care of computing the multiplicities and the order of the corresponding cyclic groups, especially when the quotient singularity is not in a normalized form.

\section{Zeta Function of Not-Well-Defined Functions}\label{not-well-defined}

In what follows the monodromy zeta function associated with not-well-defined functions over $M = X(\bd;A)$ is studied. Assume $f\in \C[x_1,\ldots,x_n]$ is a polynomial such that the following condition holds, $P\in \C^n$,
$$
  f(P) = 0 \ \Longrightarrow \ f(\bxi_\bd \cdot P) = 0,\ \forall \bxi_\bd \in \mu_\bd.
$$
Then the zero-set $\{ [{\bf x}]\in M \mid f({\bf x}) = 0 \} =: \{f=0\} \subset M$ is well defined, although~$f$~may not induce a function over $M$.

\begin{prop}\label{prop_not_well-defined}
Let $f\in \C[x_1,\ldots,x_n]$ be a reduced polynomial. The following conditions are equivalent:
\begin{enumerate}
\item $\forall P \in \C^n,\, \big[ f(P) = 0 \, \Longrightarrow \, f(\bxi_\bd \cdot P) = 0,\, \forall \bxi_\bd \in \mu_\bd \big]$.
\item $\exists {\bf v} \in \N^r$ such that $f(\bxi_\bd \cdot {\bf x}) = \bxi_\bd^{\bf v} f({\bf x}),\, \forall \bxi_\bd \in \mu_\bd$.
\item $\exists k \geq 1$ such that $f^k := f \cdot \stackrel{(k)}{\ldots} \cdot f : M \to \C$ is a function.
\end{enumerate}
\end{prop}

\begin{proof}
The only non-trivial part is $(1) \Rightarrow (2)$. Define $g_i({\bf x})$ for each $i=1,\ldots,r$ to be the polynomial $g_i({\bf x}) := f((1,\ldots,\zeta_i,\ldots,1)\cdot {\bf x}) = f(\zeta_i \cdot {\bf x})$,
where $\zeta_i$ is a fixed primitive $d_i$-th root of unity. By $(1)$, since $f$ is reduced, one has
$g_i \in IV(f) = \sqrt{f} = \langle f \rangle$.

There exists $h_i\in \C[{\bf x}]$ such that $g_i = h_i f$. Taking degrees the polynomials $h_i$'s must be constants. But,
$$
f({\bf x}) = f( \zeta_i^{d_i} \cdot {\bf x}) = g_i ( \zeta_i^{d_i-1} \cdot {\bf x} )  = h_i \cdot f( \zeta_i^{d_i-1} \cdot {\bf x}) = \cdots = h_i^{d_i} \cdot f({\bf x}).
$$

Hence $h_i = \zeta_i^{v_i}$ for some $v_i \in \N$. Now the vector ${\bf v} = (v_1,\ldots,v_r) \in \N^r$ satisfies $(2)$ and the claim follows.
\end{proof}

This example shows that the reduceness condition in the statement of the previous result is necessary.

\begin{ex}\label{bad_exam}
Let $f = (x^2+y) (x^2-y)^3 \in \C[x,y]$ and consider the cyclic quotient space $M = X(2;1,1)$. Then $\{ f = 0 \} \subset M$ defines a zero-set but there is no $k$ such that $f^k$ is a function over $M$.
\end{ex}

If $f: X(\bd; A) \to \C$ is a well-defined function, using A'Campo's formula, one easily sees that $Z(f^k; t) = Z(f; t^k)$. Therefore, when $f$ is not a function but $f^k$ is, it is natural to define the monodromy zeta function of $f$ as follows 
$$
  Z(f;t) := Z(f^k; t^{\frac{1}{k}}).
$$

One can prove that it is well defined, that is, it does not depend on $k$. Indeed, assume that $f^\l$ also induces a function over $M$, for some $\l\geq 1$. Using B\'{e}zout's identity for $k,l$ one has that $f^{\gcd(k,l)} : M \to \C$ is a function too. Denote $e:=\gcd(k,l)$, $k = k_1 e$, and $\l = \l_1 e$. Then,
$$
  Z(f^k; t^{\frac{1}{k}}) = Z(f^{k_1 e}; t^{\frac{1}{k_1 e}}) = Z(f^e; t^{\frac{1}{e}}) = 
  Z(f^{\l_1 e}; t^{\frac{1}{\l_1 e}}) = Z(f^\l; t^{\frac{1}{\l}}).
$$

The zeta function defined is a rational function on $\C[t^\frac{1}{k}]$, where $k$ is the minimum $\l\geq 1$ such that $f^\l$ is a function over $M$. When $f$ itself is a function, that is $k=1$, then it is a rational function on $\C[t]$ as usual.

The Euler characteristic of the Milnor fiber and the Milnor number are taken by definition as
$$
\chi_f := \deg Z(f;t), \qquad
\mu_f := (-1)^n [-1 + \chi_f ],
$$
where the degree of $t^{i/k}$ is $i/k$. They are in general rational numbers and they verify
$$
  \chi_f = \frac{\chi_{f^k}}{k},\qquad \mu_f = \frac{(-1)^n [1-k] + \mu_{f^k}}{k}.
$$

In this situation, our generalized A'Campo's formula can be applied directly to $f$, that is, without going through $f^k$. Note that in this case, the numbers $m_{i,j}$'s of Theorem~\ref{ATH} are rational numbers.

Let us see an example.

\begin{ex}
Let $f = x^a y^b (x^2 + y^3) \in \C[x,y]$. Consider $M = X(d;p,q)$ not necessarily written in a normalized form but assume $\gcd(d,p,q) = 1$ and $d|(2p-3q)$ hold. Then, $f$ defines a zero-set but does not induce a function over $M$.

Figure~\ref{fig_acampo7} represents an embedded $\Q$-resolution of $\{f=0\} \subset M$ that has been obtained with the blowing-up at the origin of type $\big( \frac{3}{\gcd(d,p)},\frac{2}{\gcd(d,q)} \big)$. The numbers in brackets are the order of the cyclic groups after normalizing and the others are the multiplicities of the corresponding divisors.

\begin{figure}[h t]
\centering\includegraphics{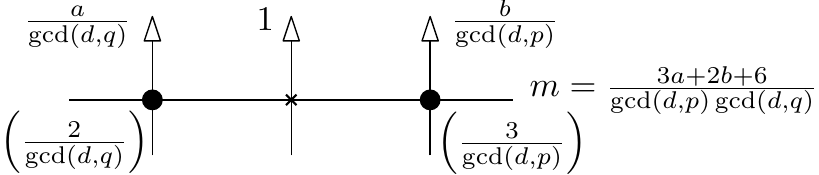}
\caption{Embedded $\Q$-resolution of $\{ x^a y^b (x^2 + y^3) = 0 \} \subset X(d;p,q)$.}
\label{fig_acampo7}
\end{figure}

Hence the monodromy zeta function is $Z(t) = (1-t^m)^{-1}$, $\chi_f = -m$, and the Milnor number is~$\mu_f = m+1$. Here $a,b$ are assumed to be non-zero, since otherwise the singular points of the final total space would also contribute to $Z(f; t)$. Some special values for $\mu_f$ are shown.

\begin{center}
\begin{tabular}{|c|c|c|c|}
\hline
$(d,p,q)$ & $(6,3,2)$ & $(1,-,-)$ & $(6,3,2)$\\
$(a,b)$ & $(2,3)$ & $(1,1)$ & $(1,1)$\\
\hline
$\mu_f$ & 4 & 12 & $17/6$\\
\hline
\end{tabular}
\end{center}

Observe that the first two values correspond to the functions $x y(x+y)$ and $x y(x^2+y^3)$ defining over~$\C^2$.
\end{ex}

\begin{remark}
In the previous example the quotient space $X(d;p,q)$ can be normalized to $X \big(\frac{d}{\gcd(d,p) \gcd(d,q)}, \frac{p}{\gcd(d,p)}, \frac{q}{\gcd(d,q)} \big)$. Under this isomorphism the polynomial $f = x^a y^b (x^2 + y^3)$ is sent to
$$
  x^{\frac{a}{\gcd(d,q)}} \cdot y^{\frac{b}{\gcd(d,p)}} \big( x^{\frac{2}{\gcd(d,q)}} + y^{\frac{3}{\gcd(d,p)}} \big),
$$
which is not a polynomial in general. This seems to force one to work with non-normalized spaces. However, since $d | (2p - 3q)$ and $\gcd(d,p,q) = 1$, then $\gcd(d,q) | 2$ and $\gcd(d,p) | 3$. Thus the previous expression is a polynomial times a monomial with rational exponents.
\end{remark}

This fact is not a coincidence as the following result clarifies. Although it can be stated in a more general setting, to simplify the ideas, we only consider polynomials in two variables over cyclic quotient singularities.

\begin{prop}
Let $d,p,q$ be three integers with $\gcd(d,p,q)=1$. Let $f(x,y) \in \C[x,y]$ such that $f(\xi_d^p x, \xi_d^q y) = \xi_d^v f(x,y)$. If $x\nmid f(x,y)$ and $y\nmid f(x,y)$, then $f(x^{1/\gcd(d,q)},y^{1/\gcd(d,p)})$ is again a polynomial.

In particular, an arbitrary polynomial $g(x,y)$ satisfying $g(\xi_d^p x, \xi_d^q y) = \xi_d^v g(x,y)$, is converted after normalizing $X(d;p,q)$ into a polynomial times a monomial with rational exponents, that is, it can be written in the form
$$
  g \big( x^{\frac{1}{\gcd(d,q)}}, y^{\frac{1}{\gcd(d,p)}} \big) = x^a y^b h(x,y),
$$
where $h(x,y)\in \C[x,y]$ and $a,b \in \mathbb{Q}_{\geq 0}$.
\end{prop}

\begin{proof}
Since $y\nmid f(x,y)$, there exists $k'\geq 0$ such that $x^{k'}$ is a monomial of~$f$. The action is diagonal and does not change the form of the monomials. Hence $x^{k'}$ has the same behavior with respect to the action as $f$, that is, $\xi_d^{k'p} x^{k'} = \xi_d^v x^{k'}$. This implies that $d | (k'p-v)$. Take $k\geq 0$ such that $k\equiv -k'$ modulo $d$.

Now $x^k f(x,y) : X(d; p,q) \to \C$ is a function with $x\nmid f(x,y)$. Then $\gcd(d,q) | k$ and $f(x^{1/\gcd(d,q)},y)$ is a polynomial.

By symmetry $f(x, y^{1/\gcd(d,p)})$ is a polynomial too and the proof is complete.
\end{proof}

As for weighted projective plane, let $F\in \C[x,y,z]$ be a $(p,q,r)$-quasi-homogeneous polynomial with $\gcd(p,q,r) = 1$. The monodromy zeta function of $F(x,y,z)$ at a point of the form $[a:b:1]$ is defined by
$$
  Z \big( F(x,y,z),[a:b:1];\, t \big) := Z \big( f(x,y,1),(a,b);\, t \big).
$$
Note that $f(\xi_r^p x, \xi_r^q, 1) = \xi_r^{\deg f} f(x,y,1)$ and thus $f(x,y,1)$ satisfies the conditions of Proposition~\ref{prop_not_well-defined}(2), where the quotient space is simply $M=X(r;p,q)$. Therefore the previous expression equals
$$
Z(f(x,y,1)^r, (a,b);\, t^{1/r}).
$$

Analogously one defines the zeta function at every point of $\P^2(p,q,r)$ and one sees that it is independent of the chosen chart. This can be generalized to spaces like $\P^n_\w/\mu$, where $\mu$ is an abelian finite group acting diagonally as usual.

To define the monodromy zeta function for polynomials defining a zero-set but there is no $k$ such that $f^k$ is a function over the quotient space, one could use A'Campo's formula and try to prove that the rational function obtained is independent of the chosen embedded $\Q$-resolution. We do not insist on the veracity of this fact because it is not the purpose of this work.

\begin{ex}
We continue here with Example~\ref{bad_exam}. Blowing up the origin of $X(2;1,1)$ with weights $(1,2)$, an embedded $\Q$-resolution of $\{f=0\} \subset X(2;1,1)$ is computed and it make sense to define the zeta function using this resolution.
\begin{figure}[h t]
\centering\includegraphics{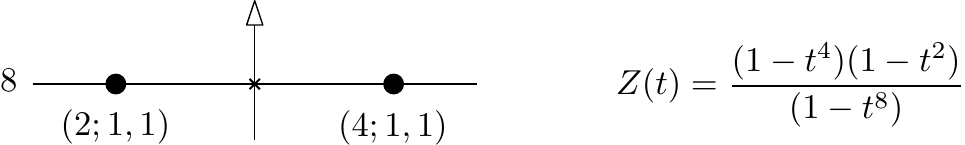}
\caption{Embedded $\Q$-resolution of $\{ (x^2  + y) (x^2 - y)^3 = 0 \} \subset X(2;1,1)$ and its monodromy zeta function.}
\end{figure}
\end{ex}

\section{Why Abelian? $D_4$ as a Quotient Singularity}\label{sec_non-abelian}

All over the paper, the ambient space $X$ is assumed to be $\C^{n}/G$, where $G$ is an abelian finite subgroup of $GL(n, \C)$. In this final part, using $D_4$ as a quotient singularity, it is exemplified the behavior for non-abelian groups. As we shall see, double points in an embedded $\Q$-resolution of a well-defined function $f: X \to \C$ contributes in general to its monodromy zeta function. In this sense abelian groups are the largest family for which Theorem~\ref{ATH} applies.

Let $\C^2$ with coordinate $(x,y)$ and consider the subgroup of $GL(2,\C)$ generated by the matrices
$$
  A = \begin{pmatrix} i & 0 \\ 0 & -i \end{pmatrix}, \qquad
  B = \begin{pmatrix} 0 & -1 \\ 1 & 0 \end{pmatrix}.
$$
Thus $A^2 = B^2 = (AB)^2 = -Id_2$. This group of order $8$, often denoted by ${\rm BD}_8$, is called the {\em binary dihedral group}. The quotient singularity $\C^2 / {\rm BD}_8$ is denoted by $D_4$.

Let us compute the zeta function of $f := (x y)^m : D_4 \to \C$, where $m$ is an even positive integer so that the map is well defined. Consider $\pi: \widehat{\C}^2 \to \C^2$ the usual blow-up at the origin. The action ${\rm BD}_8$ on $\C^2$ extends naturally to an action on $\widehat{\C}^2$ such that the induced map
$\bar{\pi}\,:\, \widehat{\C}^2 / {\rm BD}_8 \rightarrow \C^2 / {\rm BD}_8 =: D_4$
defines an embedded $\Q$-resolution of $\{ f = 0 \} \subset D_4$.

More precisely, there are three quotient singular points all of them of type $(2;1,1)$ located at the exceptional divisor. They correspond to the points~$[0:1]$, $[1:1]$, $[i:1]$ $\in \P^1 / {\rm BD}_8$. The strict transform intersects transversally the exceptional divisor at $P:=((0,0),[0:1])$ and the equation of the total transform at this point is given by $ x^m y^m : X(2;1,1) \rightarrow \C$, see Figure~\ref{fig_acampo9}.

\begin{figure}[h t]
\centering\includegraphics{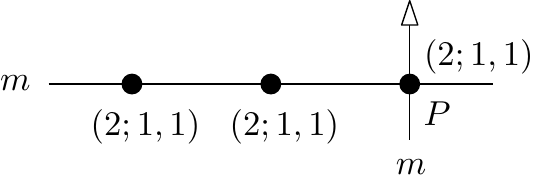}
\caption{Embedded $\Q$-resolution of $\{ (xy)^m = 0 \} \subset D_4$.}
\label{fig_acampo9}
\end{figure}

From Theorem~\ref{ATH}, the monodromy zeta function of $f$ and the Euler characteristic of the Milnor fiber are
$$
Z(t) = \frac{(1-t^{m/2})^2}{1-t^m} = \frac{1-t^{m/2}}{1+t^{m/2}}, \qquad
\chi (F) = \deg Z(t) = 0.
$$
In particular, $Z(t)$ is not trivial although $f$ defines a ``double point'' on $D_4$, as claimed.


\begin{center}
\begin{large}
\textbf{Conclusion and Future Work}
\end{large}
\end{center}

The combinatorial and computational complexity of embedded $\mathbf{Q}$-reso\-lutions is much simpler than the one of the classical embedded resolutions, but they keep as much information as needed for the comprehension of the topology of the singularity. This will become clear in the author's PhD thesis. We will prove in a forthcoming paper another advantages of these embedded $\mathbf{Q}$-resolutions, e.g.~in the computation of abstract resolutions of surfaces via Jung method, see~\cite{AMO11a, AMO11b}.

\def\cprime{$'$}

\end{document}